\documentclass{article}
\usepackage[
	letterpaper,
	left=1.15in,
	right=1.15in,
	top=0.9in,
	bottom=0.95in
]{geometry}
\usepackage{graphicx}

\usepackage{mathtools, amssymb, amsfonts, amsthm}
\usepackage{mathdots}
\usepackage{subcaption}
\usepackage{xcolor}
\usepackage{enumitem}
\usepackage[numbers]{natbib}

\usepackage{algorithm}
\usepackage{algpseudocode}

\usepackage{hyperref,xurl}
\usepackage{colortbl}
\definecolor{niceblue}{rgb}{0.0,0.19,0.56}

\usepackage{threeparttable}
\usepackage{pifont}

\newcommand{\algname}[1]{\textsc{#1}}
\definecolor{bgcolor2}{RGB}{235,245,255}

\usepackage[capitalise]{cleveref}

\makeatletter
\renewcommand*\env@matrix[1][*\c@MaxMatrixCols c]{\hskip -\arraycolsep \let\@ifnextchar\new@ifnextchar \array{#1}}
\makeatother

\usepackage{array,booktabs,multirow,makecell}

\usepackage{nicefrac}

\allowdisplaybreaks
\title{Dual-Anchor Acceleration Is Near-Optimal for Stochastic Monotone Root-Finding}

\author{TaeHo Yoon\textsuperscript{*}
\and Nicolas Loizou\textsuperscript{*}
}

\date{}

\usepackage{amsmath,amsfonts,amssymb,amsthm,bm,xcolor,mdframed}

\def\1{\bm{1}}

\DeclareMathAlphabet{\mathsfit}{\encodingdefault}{\sfdefault}{m}{sl}
\SetMathAlphabet{\mathsfit}{bold}{\encodingdefault}{\sfdefault}{bx}{n}

\newcommand{\cB}{{\mathcal{B}}}

\newcommand{\cO}{{\mathcal{O}}}

\usepackage[bb=boondox]{mathalfa}

\usepackage{xparse}
\DeclareFontFamily{U}{ntxmia}{}
\DeclareFontShape{U}{ntxmia}{m}{it}{<-> ntxmia }{}
\DeclareFontShape{U}{ntxmia}{b}{it}{<-> ntxbmia }{}
\DeclareSymbolFont{lettersA}{U}{ntxmia}{m}{it}
\SetSymbolFont{lettersA}{bold}{U}{ntxmia}{b}{it}

\ExplSyntaxOn
\NewDocumentCommand{\varmathbb}{m}
 {
  \tl_map_inline:nn { #1 }
   {
    \use:c { varbb##1 }
   }
 }
\tl_map_inline:nn { ABCDEFGHIJKLMNOPQRSTUVWXYZ }
 {
  \exp_args:Nc \DeclareMathSymbol{varbb#1}{\mathord}{lettersA}{\int_eval:n { `#1+67 }}
 }
\exp_args:Nc \DeclareMathSymbol{varbbk}{\mathord}{lettersA}{169}
\ExplSyntaxOff

\newcommand{\reals}{\mathbb{R}}

\newcommand{\opF}{{\varmathbb{F}}}
\newcommand{\opG}{{\varmathbb{G}}}
\newcommand{\opH}{{\varmathbb{H}}}
\newcommand{\opI}{{\varmathbb{I}}}
\newcommand{\opJ}{{\varmathbb{J}}}

\newcommand{\opS}{{\varmathbb{S}}}
\newcommand{\opT}{{\varmathbb{T}}}

\newcommand{\inprod}[2]{\left\langle #1,#2 \right\rangle}

\newcommand{\sqnorm}[1]{\left\| #1 \right\|^2}
\newcommand{\norm}[1]{\left\| #1 \right\|}

\newcommand{\expec}[1]{\mathbb{E}\left[ #1 \right]}
\newcommand{\condexp}[2]{\mathbb{E}\left[ #1 \,\middle|\, #2 \right]}

\usepackage{thmtools}
\definecolor{shadecolor}{gray}{0.9}
\declaretheoremstyle[
headfont=\normalfont\bfseries,
notefont=\mdseries, notebraces={(}{)},
bodyfont=\normalfont,
postheadspace=0.5em,
spaceabove=\topsep,
mdframed={
  skipabove=8pt,
  skipbelow=8pt,
  hidealllines=true,
  backgroundcolor={shadecolor},
  innerleftmargin=4pt,
  innerrightmargin=4pt}
]{shaded}

\declaretheorem[style=shaded,within=section]{definition}
\declaretheorem[style=shaded,sibling=definition]{theorem}

\declaretheorem[style=shaded,sibling=definition]{assumption}
\declaretheorem[style=shaded,sibling=definition]{corollary}

\declaretheorem[style=shaded,sibling=definition]{lemma}

\declaretheorem[style=shaded,sibling=definition]{condition}

\mdfsetup{nobreak=true}
\makeatletter
\@for\thmtype:=definition,theorem,proposition,assumption,corollary,conjecture,lemma,remark,example,condition\do{\expandafter\def\csname theH\thmtype\endcsname{\thesection.\arabic{definition}}}
\makeatother

\begin{document}

\renewcommand\thefootnote{\fnsymbol{footnote}}
\footnotetext[1]{Department of Applied Mathematics and Statistics \& Mathematical Institute for Data Science, Johns Hopkins University.
\texttt{\{tyoon7, nloizou\}@jhu.edu}}

\renewcommand\thefootnote{\arabic{footnote}}
\setcounter{footnote}{0}

\maketitle

\begin{abstract}
Among distinct optimal acceleration mechanisms for deterministic monotone root-finding problems and fixed-point problems, dual-anchoring has recently been shown to admit a more robust direct stochastic extension than standard anchor acceleration.
However, without additional strong monotonicity, the existing stochastic dual-anchoring guarantee by \citet{yoonDirectAccelerationStochastic2026} has two limitations: first, it requires cocoercivity in expectation, and second, it attains only $\cO (\epsilon^{-3})$ oracle complexity, leaving a gap to the near-optimal $\widetilde{\cO} (\epsilon^{-2})$ complexity achieved by other methods.
In this work, we address both of these limitations by combining dual-anchoring with stochastic resolvent approximation and optimized variance control.
For unbiased stochastic oracles with variance bounded by $\sigma^2$, where sample operators are monotone and uniformly $L$-Lipschitz, our algorithm finds a point $x_\epsilon$ satisfying $\expec{\norm{\opF (x_\epsilon)}} \le \epsilon$ with a near-optimal oracle complexity of $\cO ( (LD / \epsilon) \ell + (\sigma^2 / \epsilon^2) \ell^2)$, where $\ell = \log (1 + LD / \epsilon)$ and $D$ is the initial distance to a solution.
This result improves the best known oracle complexity in the noise-dominated regime under these samplewise assumptions, reducing the poly-logarithmic factor from cubic to quadratic.
\end{abstract}

\section{Introduction}
\label{sec:introduction}

We consider stochastic root-finding problem for an operator $\opF \colon \reals^d \to \reals^d$:
\begin{equation}
\label{eq:problem}
	\text{find }x_\star \in \reals^d\quad\text{such that}\quad \opF (x_\star) = 0
\end{equation}
where $\opF (x) = \expec{\opF (x; \xi)}$ is monotone.
This formulation includes convex minimization \citep{BauschkeCombettes2017_convex, RyuYin2022_largescale}, convex-concave minimax optimization \citep{nemirovskiProxmethodRateConvergence2004, nesterovDualExtrapolationIts2007, juditskySolvingVariationalInequalities2011, LinJinJordan2020_nearoptimal}, fixed-point problems \citep{Krasnoselskii1955_two, Mann1953_mean, Halpern1967_fixed, Wittmann1992_approximation} or multiplayer game theory \citep{FacchineiPang2003_finitedimensional, scutariRealComplexMonotone2014}.
Our goal is to efficiently find a point with small operator norm $\norm{\opF (\cdot)}$, using only an unbiased stochastic oracle that returns a sample operator $\opF (x; \xi)$.
Here, the efficiency is measured by the number of oracle accesses required to achieve $\expec{\norm{\opF (x)}} \le \epsilon$.

In deterministic root-finding, Halpern-type or anchoring algorithms reduce the operator norm at the accelerated, optimal rate $\norm{\opF (x_N)} = \cO (1 / N)$ \citep{Kim2021_accelerated, Diakonikolas2020_halpern, yoonAcceleratedAlgorithmsSmooth2021}.
However, the acceleration achieved by these algorithms does not directly transfer to stochastic settings due to accumulation of noise.
Existing work that attempted to extend anchor acceleration to stochastic settings used diminishing variance assumption or variance reduction techniques \citep{LeeKim2021_fast, cai2022stochastic, caiVarianceReducedHalpern2024,diakonikolas2026solving}.
On the other hand, another type of deterministic acceleration mechanism called dual-anchoring \citep{yoonOptimalAccelerationMinimax2024} has been recently shown to be more robust against oracle perturbation and stochastic noise, while enjoying the same optimal deterministic convergence guarantee as Halpern-type methods \citep{yoonTheoryCompositionDuality, yoonDirectAccelerationStochastic2026}.
In particular, \citet{yoonDirectAccelerationStochastic2026} showed that a simple stochastic dual-anchoring algorithm called \algname{S-Dual-OHM} with constant mini-batch size $B$ achieves
\begin{align}
\label{eqn:prior-work-key-bound}
	\expec{\norm{\opF (x_{N - 1})}}^2 \le \expec{\sqnorm{\opF (x_{N - 1})}} \le \frac{4L^2 \sqnorm{x_0 - x_\star}}{N^2} + \frac{6\sigma^2}{B}
\end{align}
for problems with cocoercive $\opF (x; \xi)$, where $\sigma^2$ denotes a bound on the oracle variance.
Hence, taking $N = \cO (\epsilon^{-1})$ and $B = \cO (\epsilon^{-2})$ gives $\cO (\epsilon^{-3})$ oracle complexity for achieving $\expec{\norm{\opF (x)}} \le \epsilon$, which already offers a clean improvement over stochastic anchoring-based approaches, which required additional algorithmic techniques to achieve the same complexity.

On the other hand, the result \eqref{eqn:prior-work-key-bound} was limited to cocoercive problems and did not cover the more general monotone stochastic root-finding problems.
Furthermore, the $\cO (\epsilon^{-3})$ complexity derived by~\eqref{eqn:prior-work-key-bound} still has a significant gap to the $\Omega (\epsilon^{-2})$ lower bound from \cite{chenNearoptimalAlgorithmsMaking2024}.
\emph{In this work, we address both of these issues and develop a conceptually simple, near-optimal algorithm for monotone stochastic problems based on dual-anchoring.}
More precisely, we provide the following technical contributions.

\subsection{Main contributions}

First, we show that \algname{S-Dual-OHM} with optimized scheduling of batch sizes $B_k$ in fact attains an improved, near-optimal complexity of $\widetilde\cO (\epsilon^{-2})$ for stochastic cocoercive root-finding (\cref{thm:fixed-point} \& \cref{cor:fixed-point}).
Interestingly, this is achieved without fundamentally changing the update rule of \algname{S-Dual-OHM} or its convergence analysis, and the algorithm remains single-loop.

Second, we extend the near-optimality of \algname{S-Dual-OHM} with optimized variance control to stochastic monotone root-finding by using subroutines for approximately computing the resolvents.
Assuming that each $\opF (\cdot; \xi)$ is monotone and $L$-Lipschitz, the proposed algorithm has the complexity
\begin{equation}
\label{eq:intro-complexity}
	\cO\! \left( \ell + \frac{L\norm{x_0 - x_\star}}{\epsilon} \ell + \frac{\sigma^2}{\epsilon^2}\ell^2\right) \text{ where } \ell = \log \left( 1 + \frac{L\norm{x_0 - x_\star}}{\epsilon} \right)
\end{equation}
(\cref{cor:weighted}).
This complexity is near-optimal, and improves upon the previous state-of-the-art algorithm \algname{RAIN} \citep{chenNearoptimalAlgorithmsMaking2024} that has cubic poly-logarithmic dependence on $\frac{L\norm{x_0 - x_\star}}{\epsilon}$ in the dominant term, although \algname{RAIN} only requires monotonicity and Lipschitzness of $\opF$.

\subsection{Related work}

\paragraph{Halpern-type acceleration and its stochastic extension.}
Halpern iteration \citep{Halpern1967_fixed} originates in the theory of fixed point problems, which find a point $x_\star$ such that $x_\star = \opT (x_\star)$ for nonexpansive $\opT \colon \reals^d \to \reals^d$; it can accelerate the residual convergence of fixed-point iterations \citep{SabachShtern2017_first, contrerasOptimalErrorBounds2023} and in fact, attain exact optimal complexity matching the lower bound without even a constant factor gap \citep{Lieder2021_convergence, ParkRyu2022_exact}:
\begin{align}
\tag{\algname{OHM}}
\label{eq:OHM}
	x_{k + 1} = \frac{1}{k + 2} x_0 + \frac{k + 1}{k + 2} \opT x_k .
\end{align}
This idea, also called anchoring in the literature, was adapted to develop accelerated algorithms for monotone root-finding and minimax optimization \citep{Diakonikolas2020_halpern, yoonAcceleratedAlgorithmsSmooth2021, LeeKim2021_fast, Tran-DinhLuo2021_halperntype, CaiZheng2023_accelerated}.
In stochastic problems, \citet{cai2022stochastic, caiVarianceReducedHalpern2024} combined Halpern iteration with variance reduction to achieve either $\cO (\epsilon^{-3})$ or improved complexity in certain finite-sum regime.
\citet{diakonikolas2026solving} improved this to $\widetilde\cO(\epsilon^{-2})$ with gradual Halpern method \citep{diakonikolas2025pushing} for samplewise nonexpansive fixed-point problems.
\citet{bravoStochasticHalpernIteration2026} studied stochastic Halpern iteration in more general normed spaces.

\paragraph{Algorithms and complexity for stochastic root-finding problems.}
Using stochastic variants of classical methods such as extragradient \citep{Korpelevich1976_extragradient} or optimistic gradient \citep{popovModificationArrowHurwiczMethod1980, RakhlinSridharan2013_online, RakhlinSridharan2013_optimization} typically yields a suboptimal $\cO (\epsilon^{-4})$ oracle complexity \citep{diakonikolasEfficientMethodsStructured2021, gorbunovStochasticExtragradientGeneral2022, choudhurySinglecallStochasticExtragradient2023}.
Extending the plain anchoring mechanism of \ref{eq:OHM} results in accumulation of errors \citep{LeeKim2021_fast}, which is why variance reduction was used in prior work \citep{cai2022stochastic, caiVarianceReducedHalpern2024} that explored stochastic Halpern-type algorithms.
\citet{chenNearoptimalAlgorithmsMaking2024} proposed the Recursive Anchored Iteration (\algname{RAIN}) algorithm together with a complexity lower bound, which uses amortized scheduling of regularization strength to achieve a near-optimal complexity $\widetilde\cO \left( \frac{LD}{\epsilon} + \frac{\sigma^2}{\epsilon^2} \right)$ for stochastic monotone root-finding problems.
\citet{yoonDirectAccelerationStochastic2026} considered an alternative approach of extending a distinct form of acceleration from deterministic fixed-point problem setting developed by \citet{yoonOptimalAccelerationMinimax2024}:
\begin{align}
\tag{\algname{Dual-OHM}}
\label{eq:Dual-OHM}
	x_{k + 1} = x_k + \frac{N - k - 1}{N - k} \left( \opT (x_k) - \opT (x_{k - 1}) \right)
\end{align}
where the total iteration number $N$ is predetermined.
They showed that \ref{eq:Dual-OHM} is more robust to stochastic noise, and $\cO (\epsilon^{-3})$ complexity can be more easily attained without variance reduction or recursive regularization as in the case of stochastic anchoring methods, albeit for a narrower class of cocoercive stochastic root-finding problems.

\paragraph{Proximal view of root-finding problems.}
Proximal-point methods provide a classical framework for solving monotone inclusions and convex-concave minimax problems using resolvent operations \citep{martinetRegularisationDinequationsVariationnelles1970, rockafellarMonotoneOperatorsProximal1976}.
Resolvents are often considered as \emph{implicit} operations that are generally not computable.
Approximating resolvents using forward/explicit evaluations of $\opF$ is a well-established algorithmic principle in the literature.
Extragradient and optimistic gradient methods admit interpretations as approximations to proximal-point iterations \citep{monteiroComplexityHybridProximal2010, monteiroIterationcomplexityNewtonProximal2012, MokhtariOzdaglarPattathil2020_unified, MokhtariOzdaglarPattathil2020_convergence, jiangGeneralizedOptimisticMethods2025}, This viewpoint was also combined with anchor acceleration to design more efficient algorithms \citep{Kim2021_accelerated, Diakonikolas2020_halpern} or to establish a formal connection between implicit and explicit algorithms \citep{SuhParkRyu2023_continuoustime, yoonAcceleratedMinimaxAlgorithms2025}.
Our algorithmic design combines the idea of approximate resolvent computation with stochastic proximal point methods \citep{bianchiErgodicConvergenceStochastic2016, sadievStochasticProximalPoint2024}.

\subsection{Preliminaries and assumptions}
\label{sec:preliminaries}

We assume throughout the paper that \eqref{eq:problem} has a solution $x_\star$.
We say $\opF$ is \emph{monotone} if $\inprod{\opF (x) - \opF (y)}{x - y} \ge 0$ for all $x, y \in \reals^d$, and $L$-\emph{Lipschitz} if $\norm{\opF (x) - \opF (y)} \le L\norm{x - y}$ for all $x, y \in \reals^d$.
It is $\nicefrac{1}{L}$-\emph{cocoercive} if $\inprod{\opF (x) - \opF (y)}{x - y} \ge \frac1L\sqnorm{\opF (x) - \opF (y)}$ for all $x, y \in \reals^d$.
While cocoercivity implies monotonicity and Lipschitz continuity, the converse need not hold.
We say an operator $\opT \colon \reals^d \to \reals^d$ is \emph{nonexpansive} if $\norm{\opT (x) - \opT (y)} \le \norm{x - y}$ for all $x, y \in \reals^d$.

Denote the identity operator by $\opI$.
For a monotone continuous $\opF$, its \emph{resolvent} $\opJ_{\alpha\opF} = (\opI + \alpha\opF)^{-1}$ is well defined operator on $\reals^d$ and is (firmly) nonexpansive for any $\alpha > 0$, and fixed points of $\opJ_{\alpha\opF}$ are precisely the zeros of $\opF$.
When $\opF$ is $\nicefrac{1}{L}$-cocoercive, $\opI - \alpha\opF$ is nonexpansive for $0 < \alpha \le \frac{2}{L}$ \citep{BauschkeCombettes2017_convex}.

\begin{assumption}[Stochastic oracle]
\label{ass:oracle}
The stochastic oracle operator $\opF (x; \xi)$ is jointly measurable, unbiased and has bounded variance.
That is, there exists $\sigma > 0$ such that for every $x \in \reals^d$,
\[
	\expec{\opF (x; \xi)} = \opF (x), \qquad \expec{\sqnorm{\opF (x; \xi) - \opF (x)}} \le \sigma^2.
\]
Furthermore, for each sampled $\xi$, the stochastic operator $\opF (\cdot; \xi)$ can be evaluated at multiple points.
\end{assumption}

At iteration $k$, we can draw a mini-batch $\cB_k$ containing $B_k$ independent samples, which defines $\opF_{\cB_k} (x) = \frac1{B_k}\sum_{\xi \in \cB_k}\opF (x; \xi)$.
Whenever we use them, mini-batches are independent across iterations, newly drawn samples being independent of all previously drawn ones.

\section{Stochastic cocoercive root-finding and fixed-point problems}
\label{sec:fixed-point}

We first consider stochastic root-finding problems with cocoercive $\opF = \expec{\opF (\cdot; \xi)}$.
This can be recast into stochastic fixed-point problem by taking $\opT (\cdot; \xi) = \opI - \alpha \opF (\cdot; \xi)$ for appropriate $\alpha$, so that $\opT = \expec{\opT (\cdot; \xi)}$ is nonexpansive.
For the analysis, we will take this fixed-point view.
Throughout the section, $x_\star$ will denote a fixed point of $\opT$, or equivalently, a zero of $\opF$.

The main result of this section is to improve the $\cO (\epsilon^{-3})$ complexity bound for \algname{S-Dual-OHM} provided in \cite{yoonDirectAccelerationStochastic2026} by optimizing the variance term contributed by each iteration number $k$.
For the convergence analysis, it is convenient to consider the following abstraction: at iteration $k$, consider using a \emph{random map} $\opT_k$ at each iteration $k$ such that $\expec{\opT_k (u)} = \opT (u)$ and $\expec{\sqnorm{\opT_k (u) - \opT (u)}} \le \sigma_k^2$ for all deterministic $u \in \reals^d$.
In particular, for any $u, v \in \reals^d$ it satisfies
\begin{equation}
\label{eq:fixed-point-interface}
	\expec{\sqnorm{\opT_k (u) - \opT (v)}} = \sqnorm{\opT (u) - \opT (v)} + \expec{\sqnorm{\opT_k (u) - \opT (u)}} \le \sqnorm{u - v} + \sigma_k^2.
\end{equation}
We will additionally require each $\opT_k$ to be square-nonexpansive in expectation.
Together, we state:

\begin{condition}
\label{ass:fixed-point}
The jointly measurable random maps $\opT_k$ are independent across iterations, and for every deterministic $u, v \in \reals^d$, satisfy
\begin{align}
\label{eq:fixed-point-stability}
	\expec{\sqnorm{\opT_k (u) - \opT (v)}} \le \sqnorm{u - v} + \sigma_k^2 , \quad \expec{\sqnorm{\opT_k (u) - \opT_k (v)}} \le \sqnorm{u - v},
\end{align}
where $\sigma_k^2 \ge 0$ is deterministic and the expectation is taken over randomness in $\opT_k$.
\end{condition}

Now fix $N \ge 2$ and consider the following formulation of \algname{S-Dual-OHM}:
\begin{equation}
\tag{\algname{S-Dual-OHM}}
\label{eq:dual}
\begin{aligned}
	x_{k + 1} & = x_k + \frac{N - k - 1}{N - k} \left( \opT_k (x_k) - \opT_{k - 1} (x_{k - 1}) \right)
\end{aligned}
\end{equation}
for $k = 0, \ldots, N - 2$, where $\opT_{-1} (x_{-1}) := x_0$.
Its final output, $x_{N - 1}$, has small fixed-point residual.

\begin{theorem}
\label{thm:fixed-point}
\ref{eq:dual} run with random maps $\opT_0, \dots, \opT_{N - 2}$ satisfying Condition~\ref{ass:fixed-point} and $x_0 \in \reals^d$ such that $\norm{x_0 - x_\star} \le D$ satisfies: with $a_{N, k} = \frac1{(N - k) (N - k - 1)}$,
\begin{equation*}
	\expec{\sqnorm{x_{N - 1} - \opT (x_{N - 1})}} \le \frac{4D^2}{N^2} + 6\sum_{k = 0}^{N - 2}a_{N, k}\sigma_k^2 .
\end{equation*}
\end{theorem}

Observe that $\sum_k a_{N, k} = 1 - \frac{1}{N}$, so if each $\opT_k$ is a mini-batch operator with a constant batch-size $B$, then we can take $\sigma_k^2 \equiv \frac{\alpha^2 \sigma^2}{B}$ and the summation of variance terms is bounded by $\frac{6\alpha^2 \sigma^2}{B}$, independently of $N$.
This case precisely recovers \cite[Theorem~4.1]{yoonDirectAccelerationStochastic2026}.
We note that while \cref{thm:fixed-point} is more general than \cite[Theorem~4.1]{yoonDirectAccelerationStochastic2026}, their proof essentially provides all ideas necessary to prove this result.

Our new insight here is that \emph{the noise terms from earlier iterations are much less costly than the later ones} because they are weighted by $a_{N, k} = \frac1{(N - k) (N - k - 1)}$.
Therefore, we may allow larger $\sigma_k$ for early $k$ and enforce smaller $\sigma_k$ later.
For the special case with mini-batching, this means that we can allocate smaller batches to the earlier iterations without affecting the total noise term much.
Optimizing this batch schedule yields a significant oracle complexity improvement, as detailed below.

\subsection{Mini-batching and optimized allocation of the oracle budget}
\label{subsec:cocoercive-batch-allocation}

Consider the case where $\opT (\cdot; \xi)$ is an unbiased stochastic oracle such that $\expec{\opT (x; \xi)} = \opT (x)$, $\expec{\sqnorm{\opT (x; \xi) - \opT (x)}} \le \tau^2$ and $\expec{\sqnorm{\opT (x; \xi) - \opT (y; \xi)}} \le \sqnorm{x - y}$ for any deterministic $x, y \in \reals^d$, and we use the mini-batch operator
\begin{equation}
\label{eq:fixed-point-batch}
	\opT_k (x) = \frac1{B_k}\sum_{\xi \in \cB_k} \opT (x; \xi)
\end{equation}
where each $\cB_k$ is a uniformly and independently sampled mini-batch of size $B_k$.
In this case, Condition~\ref{ass:fixed-point} is satisfied with $\sigma_k^2 = \frac{\tau^2}{B_k}$, and one evaluation of $\opT_k$ costs $B_k$ oracle evaluations.
The total noise term given by \cref{thm:fixed-point} is therefore proportional to $A_N := \sum_{k = 0}^{N - 2} \frac{a_{N, k}}{B_k}$.
Now treating batch-sizes as real numbers, and assuming the total oracle access (complexity) constraint $\sum_{k = 0}^{N - 2} B_k \le Q$, we can optimize $B_k$ to minimize the noise term, using Cauchy--Schwarz:
\begin{equation}
\label{eq:allocation}
	\left(\sum_{k = 0}^{N - 2} \sqrt{a_{N, k}}\right)^2 \le \left(\sum_{k = 0}^{N - 2} \frac{a_{N, k}}{B_k}\right) \left(\sum_{k = 0}^{N - 2} B_k\right).
\end{equation}
Equality holds when $B_k$ is proportional to $\sqrt{a_{N, k}}$.
Hence $B_k = \frac{Q\sqrt{a_{N, k}}}{S_N}$, where $S_N = \sum_{k = 0}^{N - 2} \sqrt{a_{N, k}}$, is the optimal allocation satisfying $A_N = \frac{S_N^2}{Q}$.
Now rounding up each $B_k$ yields
\begin{equation}
\label{eq:optimal-cocoercive}
	B_k = \left\lceil\frac{Q\sqrt{a_{N, k}}}{S_N}\right\rceil, \qquad \sum_{k = 0}^{N - 2} B_k \le Q + N - 1, \qquad \sum_{k = 0}^{N - 2} \frac{a_{N, k}}{B_k} \le \frac{S_N^2}{Q}.
\end{equation}
Note that $S_N \le \sum_{k = 0}^{N - 2} \frac{1}{N - k - 1} \le 1 + \log N$.

\begin{corollary}
\label{cor:fixed-point}
Let $\epsilon > 0$.
\ref{eq:dual} run with the mini-batch operators~\eqref{eq:fixed-point-batch}, with
\[
	N = \max\left\{2, \left\lceil \frac{4D}{\epsilon} \right\rceil\right\}, \quad 
	Q = \max\left\{1, \frac{8\tau^2S_N^2}{\epsilon^2} \right\}
\]
and the batch schedule \eqref{eq:optimal-cocoercive} attains $\expec{\sqnorm{x_{N - 1} - \opT (x_{N - 1})}} \le \epsilon^2$ using
\begin{equation}
\label{eq:fixed-point-complexity}
	\cO\! \left( \frac{D}{\epsilon} + \frac{\tau^2}{\epsilon^2} \left( 1 + \log \left( 1 + \frac{D}{\epsilon}\right) \right)^2 \right)
\end{equation}
oracle evaluations.
\end{corollary}

\begin{proof}
By \cref{thm:fixed-point} and \eqref{eq:optimal-cocoercive}, we have
\begin{align*}
	\expec{\sqnorm{x_{N - 1} - \opT (x_{N - 1})}} \le \frac{4D^2}{N^2} + 6\tau^2 \sum_{k = 0}^{N - 2} \frac{a_{N, k}}{B_k} \le \frac{4D^2}{N^2} + \frac{6\tau^2 S_N^2}{Q} \le \frac{\epsilon^2}{4} + \frac{3\epsilon^2}{4} = \epsilon^2 .
\end{align*}
Because $N \le 2 + \left\lceil\frac{4D}{\epsilon}\right\rceil$ and $Q \le 1 + \frac{8\tau^2 S_N^2}{\epsilon^2} \le 1 + \frac{8\tau^2 (1 + \log N)^2}{\epsilon^2}$, the total complexity is at most
\begin{align*}
	\sum_{k = 0}^{N - 2} B_k \le Q + N - 1 \le 2 + \left\lceil\frac{4D}{\epsilon}\right\rceil + \frac{8\tau^2}{\epsilon^2} \left( 1 + \log \left( 2 + \left\lceil\frac{4D}{\epsilon}\right\rceil \right) \right)^2 .
\end{align*}

\end{proof}

\subsection{Specialization to stochastic cocoercive root-finding}

Consider the operator oracle of Assumption~\ref{ass:oracle}.
Suppose additionally that $\opF (\cdot; \xi)$ is cocoercive in expectation, i.e., for every $u, v \in \reals^d$,
\begin{equation}
\label{eq:cocoercivity-in-expectation}
	\expec{\sqnorm{\opF (u; \xi) - \opF (v; \xi)}} \le L\inprod{\opF (u) - \opF (v)}{u - v}.
\end{equation}
For $0 < \alpha \le \nicefrac{2}{L}$, taking $\opT = \opI - \alpha\opF$, $\opT (\cdot; \xi) = \opI - \alpha\opF (\cdot; \xi)$ and $\opT_k = \opI - \alpha \opF_{\cB_k}$ satisfies Condition~\ref{ass:fixed-point} with $\tau^2 = \alpha^2 \sigma^2$ and $\sigma_k^2 = \frac{\alpha^2 \sigma^2}{B_k}$, and Theorem~\ref{thm:fixed-point} gives
\begin{equation*}
	\expec{\sqnorm{\opF (x_{N - 1})}} \le \frac{4D^2}{\alpha^2N^2} + 6\sigma^2\sum_{k = 0}^{N - 2}\frac{a_{N, k}}{B_k}.
\end{equation*}
With batch-sizes~\eqref{eq:optimal-cocoercive} and $\alpha = \frac{1}{L}$, Corollary~\ref{cor:fixed-point} with $\alpha\epsilon$ in place of $\epsilon$ yields $\expec{\sqnorm{\opF (x_{N - 1})}} \le \epsilon^2$ in
\begin{equation*}
	\cO\! \left(\frac{LD}{\epsilon} + \frac{\sigma^2}{\epsilon^2}\log^2 \left(\frac{LD}{\epsilon}\right) \right)
\end{equation*}
operator evaluations.

\section{Stochastic monotone root-finding}
\label{sec:monotone}

Now we extend the result of the previous section to stochastic root-finding for monotone and Lipschitz operators.
Recall that the resolvent $\opT = \opJ_{\alpha\opF}$ of a maximally monotone operator $\opF$ is nonexpansive; we leverage this connection to again take the fixed-point view as in \cref{sec:fixed-point}.
Because resolvents are not \emph{exactly} computable, we run inner loops to compute them \emph{approximately}.
These will yield random maps $\opT_k$, which may no longer be unbiased estimators of $\opT$ satisfying Condition~\ref{ass:fixed-point}.
We instead require the following weaker condition.
Later, we show how to select concrete $\opT_k$ satisfying it.

\begin{condition}
\label{ass:monotone-Tk}
The jointly measurable random maps $\opT_k$ are independent across iterations, and for every deterministic $u, v \in \reals^d$, satisfy
\begin{align}
	\expec{\sqnorm{\opT_k (u) - \opT (v)}} & \le \sqnorm{u - v} + \rho_k \left( \sqnorm{u - x_\star} + \sqnorm{v - x_\star} \right) + \sigma_k^2 \label{eq:monotone-bias} \\
	\expec{\sqnorm{\opT_k (u) - \opT_k (v)}} & \le \sqnorm{u - v} + \rho_k \left( \sqnorm{u - x_\star} + \sqnorm{v - x_\star} + \sigma_k^2 \right) \label{eq:monotone-stability}
\end{align}
where $\sigma_k^2, \rho_k \ge 0$ are deterministic and the expectation is taken over randomness in $\opT_k$.
\end{condition}

\subsection{Outer bound under generalized condition on \texorpdfstring{$\opT_k$}{Tk}}
\label{subsec:monotone-outer}

Now we present an analogue of \cref{thm:fixed-point} under the relaxed Condition~\ref{ass:monotone-Tk}.
This result exactly recovers \cref{thm:fixed-point} as a special case with $\rho_k \equiv 0$, so the proof of \cref{thm:outer-fixed-point}, which we present in Appendix~\ref{app:fixed-point}, accounts for the results from the previous section.

\begin{theorem}
\label{thm:outer-fixed-point}
\ref{eq:dual} run with random maps $\opT_0, \dots, \opT_{N - 2}$ satisfying Condition~\ref{ass:monotone-Tk} and $x_0 \in \reals^d$ such that $\norm{x_0 - x_\star} \le D$ satisfies: with $a_{N, k} = \frac1{(N - k) (N - k - 1)}$ and $R = \sum_{k = 0}^{N - 2} \rho_k$,
\begin{equation*}
	\expec{\sqnorm{x_{N - 1} - \opT (x_{N - 1})}} \le \frac{4D^2}{N^2} + 6\sum_{k = 0}^{N - 2}a_{N, k}\sigma_k^2 + 15 Re^R \left(D^2 + \sum_{k = 0}^{N - 2} \sigma_k^2 \right) .
\end{equation*}
\end{theorem}

In particular, when $\opT = \opJ_{\alpha\opF}$, we obtain the following bound on the operator norm.

\begin{corollary}
\label{cor:outer-monotone}
Let $\opF \colon \reals^d \to \reals^d$ be monotone and $L$-Lipschitz, and $\opT = \opJ_{\alpha\opF}$ for some $\alpha \in \left(0, \frac{1}{L}\right]$.
Fix $N \ge 2$, $x_0 \in \reals^d$ and let $D = \norm{x_0 - x_\star}$.
Let \ref{eq:dual} be run with random maps $\opT_0, \dots, \opT_{N - 1}$ satisfying Condition~\ref{ass:monotone-Tk}, and set $y_N = \opT_{N - 1} (x_{N - 1})$.
Then, with $R = \sum_{k = 0}^{N - 1}\rho_k$,
\begin{equation}
\label{eq:abstract-outer-bound}
	\expec{\sqnorm{\opF (y_N)}} \le \frac{8D^2}{\alpha^2N^2} + \frac{24}{\alpha^2}\sum_{k = 0}^{N - 1}\frac{\sigma_k^2}{(N - k)^2} + \frac{30}{\alpha^2} Re^R \left(D^2 + \sum_{k = 0}^{N - 1} \sigma_k^2 \right) .
\end{equation}
\end{corollary}

\subsection{Inner loop for stable approximate resolvents \texorpdfstring{$\opT_k$}{Tk}}

Our choice of inner loop algorithm, i.e., the construction of $\opT_k$, will require the following samplewise monotonicity and Lipschitzness in addition to Assumption~\ref{ass:oracle}.
Note that it implies the monotonicity and $L$-Lipschitzness of the mean operator $\opF$, but not vice versa.

\begin{assumption}
\label{ass:samplewise}
$\opF (\cdot; \xi)$ is almost surely monotone and $L$-Lipschitz on $\reals^d$, with a common $L > 0$.
\end{assumption}

In \cref{sec:fixed-point}, we achieved reduced variance $\sigma_k^2$ by allocating more sample budgets to iteration $k$, via mini-batching.
Here we aim to choose $\opT_k$ similarly, ensuring Condition~\ref{ass:monotone-Tk} with smaller $\sigma_k^2$ given a larger number of subroutine (inner loop) iterations.
First, suppose ideally that we can compute stochastic proximal operations exactly.
Given sample budget $B_k \ge 1$, first consider $\widehat\opT_k$ defined by:
\begin{equation}
\label{eq:streaming-ideal}
\begin{aligned}
\begin{cases}
	z_0 = u \\
	z_{j + 1} = \opJ_{\alpha\gamma_j\opF (\cdot; \xi_j)} \bigl( (1 - \gamma_j) z_j + \gamma_ju\bigr) \text{ for } j = 0, \dots, B_k - 1 \\
	\widehat\opT_k (u) = z_{B_k}
\end{cases}
\end{aligned}
\end{equation}
where we use $\alpha = \frac{1}{2L}$, $\gamma_j = \frac{2}{j + 2}$ and each $\xi_j$ is a fresh sample drawn independently of all preceding ones.
Following the technical treatment of stochastic proximal point methods similar to \cite{sadievStochasticProximalPoint2024}, we can show that $\widehat\opT_k$ is almost surely nonexpansive, and additionally satisfies
\begin{equation}
\label{eq:streaming-comparison}
	\expec{\sqnorm{\widehat\opT_k (u) - \opT (v)}} \le \sqnorm{u - v} + \frac{4\alpha^2\sigma^2}{B_k + 1}
\end{equation}
for any $u, v \in \reals^d$, as we show in Appendix~\ref{app:monotone}.
Hence $\widehat\opT_k$ satisfies Condition~\ref{ass:fixed-point} with $\sigma_k^2 = \frac{4\alpha^2\sigma^2}{B_k + 1}$.

However, because $\opJ_{\alpha\gamma_j\opF (\cdot; \xi_j)} (p_j)$ is generally not directly computable, we can only compute it up to errors, via forward evaluations $\opF (\cdot; \xi_j)$.
Nevertheless, this is not a significant overhead because it is equivalent to finding the zero of a $1$-strongly monotone operator $z \mapsto \opI + \alpha \gamma_j \opF (z; \xi_j) - p_j$, which converges rapidly with simple operator descent using step-size $1$.
Algorithm~\ref{alg:resolvent} presents an approximate version of~\eqref{eq:streaming-ideal} replacing the resolvent step with this descent loop.

\begin{algorithm}[t]
\caption{Approximate resolvent $\opT_k$}
\label{alg:resolvent}
\begin{algorithmic}[1]
	\Require Input $u$, Sample budget $B_k \ge 1$, Tolerance $\rho_k \in (0, 1]$, Lipschitz constant $L$, Hyperparameters $\alpha, \gamma_j$
	\State $w_0 \gets u$,
	\For{$j = 0, \ldots, B_k - 1$}
		\State Draw $\xi_j$
		\State $p_j \gets (1 - \gamma_j) w_j + \gamma_j u$
		\State $q_0 \gets p_j$, $m_{k, j} \gets \left\lceil \frac{\log (16 / \rho_k)}{\log (j + 2)} \right\rceil$
		\For{$r = 0, \ldots, m_{k, j}$}
			\State $q_{r + 1} \gets p_j - \alpha\gamma_j\opF (q_r; \xi_j)$
		\EndFor
		\State $w_{j + 1} \gets q_{m_{k, j} + 1}$
	\EndFor
	\State $\opT_k (u) \gets w_{B_k}$
	\State \Return $\opT_k (u)$
\end{algorithmic}
\end{algorithm}

\begin{lemma}
\label{lem:actual-resolvent}
Under Assumptions~\ref{ass:oracle} and~\ref{ass:samplewise}, the random map $\opT_k$ defined via \cref{alg:resolvent} with $\alpha = \frac{1}{2L}$ and $\gamma_j = \frac{2}{j + 2}$ satisfies Condition~\ref{ass:monotone-Tk} with $\sigma_k^2 = \frac{6\alpha^2\sigma^2}{B_k + 1}$ and uses $\cO \left(B_k \left[1 + \frac{\log (16 / \rho_k)}{\log (B_k + 1)}\right] \right)$ oracle evaluations.
\end{lemma}

The parameter $\rho_k$ appears in Condition~\ref{ass:monotone-Tk} and controls the tolerance to violation of nonexpansiveness in expectation.
$B_k$ is the sample budget, controlling the variance bound $\sigma_k^2$.

\subsection{Optimized allocation of \texorpdfstring{$B_k$}{Bk} and complexity analysis}

Fix $N \ge 2$ and choose $\rho_k \equiv \frac{1}{2N^3}$, so that $Re^R = \cO \left(\frac{1}{N^2}\right)$ in \cref{cor:outer-monotone}, so together with \cref{lem:actual-resolvent}, the total variance term in~\eqref{eq:abstract-outer-bound} is bounded by a constant multiple of $A_N := \sum_{k = 0}^{N - 1} \frac{1}{(B_k + 1) (N - k)^2} \le \sum_{k = 0}^{N - 1} \frac{1}{B_k (N - k)^2}$.
\cref{lem:actual-resolvent} also shows that the total complexity is
\begin{align*}
	\cO \left( \sum_{k = 0}^{N - 1} B_k + \log (eN) \sum_{k = 0}^{N - 1} \frac{B_k}{\log (B_k + 1)} \right) = \cO \left(\log (eN) \left[N + \sum_{k = 0}^{N - 1} \frac{B_k}{\log (eN / (N - k))}\right] \right)
\end{align*}
where we show in Appendix~\ref{app:allocation-streaming} that the latter bound holds for any $B_k \ge 1$.
Now to optimize the allocation of $B_k$, similarly as in Section~\ref{subsec:cocoercive-batch-allocation}, we temporarily view $B_k$ as real numbers and use Cauchy-Schwarz to minimize $A_N$ while keeping $Q := \sum_{k = 0}^{N - 1} \frac{B_k}{\log (eN / (N - k))}$ fixed:
\[
	\left(\sum_{k = 0}^{N - 1} \frac1{(N - k) \sqrt{\log (eN / (N - k))}}\right)^2 \le \left(\sum_{k = 0}^{N - 1}\frac1{(N - k)^2B_k}\right) \left(\sum_{k = 0}^{N - 1}\frac{B_k}{\log (eN / (N - k))}\right).
\]
Equality holds when $B_k$ is proportional to $\frac{1}{N - k} \sqrt{\log \frac{eN}{N - k}}$.
More precisely, we choose:
\begin{equation}
\label{eq:weighted-monotone}
	B_k = \max\! \left\{1, \left\lceil \frac{1024\sigma^2}{\epsilon^2 (N - k)} \sqrt{\log (eN) \log\! \left(\frac{eN}{N - k}\right)} \right\rceil \right\}
\end{equation}
for $k = 0, \dots, N - 1$.
The resulting complexity bound is as follows, and we present its detailed proof in Appendix~\ref{app:allocation-streaming}.

\begin{corollary}[Oracle complexity for monotone root-finding]
\label{cor:weighted}
Under Assumptions~\ref{ass:oracle} and~\ref{ass:samplewise}, let $\epsilon > 0$, $\norm{x_0 - x_\star} \le D$.
Choose $N = \max \left\{2, \left\lceil \frac{32LD}{\epsilon}\right\rceil \right\}$ and $\rho_k \equiv \frac{1}{2N^3}$.
Run \ref{eq:dual} with $\opT_k$ computed by Algorithm~\ref{alg:resolvent}, using batch schedule \eqref{eq:weighted-monotone}.
Then its final output $y_N = \opT_{N - 1} (x_{N - 1})$ attains $\expec{\sqnorm{\opF (y_N)}} \le \epsilon^2$ using
\begin{equation}
\label{eq:streaming-complexity}
	\cO\! \left( \left(1 + \frac{LD}{\epsilon}\right) \log \left(1 + \frac{LD}{\epsilon}\right) + \frac{\sigma^2}{\epsilon^2} \log^2 \left( 1 + \frac{LD}{\epsilon} \right) \right)
\end{equation}
operator evaluations.
\end{corollary}

\section{Numerical experiments}
\label{sec:experiments}

We evaluate the empirical effectiveness of \ref{eq:dual} for cocoercive and monotone stochastic root-finding problems.
We consider two problems, both with stochastic oracle of the form $\opF (x; \xi) = \opF (x) + \xi$ where $\xi \sim \mathcal N \left( 0, \frac{\sigma^2 I_d}{d} \right)$, and use the residual norm $\norm{\opF (x_k)}$ as a measure of convergence.

\paragraph{Baselines.}
Denote by $\cB_k$ and $\cB_k'$ independent mini-batches of fixed size $B$, and let $\opF_{\cB_k}$ and $\opF_{\cB_k'}$ be the corresponding mini-batch operators.
We consider the following basic algorithms: Stochastic Gradient Descent-Ascent (\algname{SGDA}), Stochastic \algname{OHM} (\algname{S-OHM}) and Stochastic Extragradient (\algname{SEG}), where the first two algorithms are used in Experiment~1 where $\opF$ is cocoercive, and \algname{SEG} is used for both experiments.
Their update rules are given by
\[
\begin{aligned}
	\text{\algname{SGDA}}:\quad & x_{k + 1} = x_k - \eta\opF_{\cB_k} (x_k) \\
	\text{\algname{SEG}}:\quad & x_{k + 1 / 2} = x_k - \eta\opF_{\cB_k} (x_k), \qquad x_{k + 1} = x_k - \eta\opF_{\cB_k'} (x_{k + 1 / 2}) \\
	\text{\algname{S-OHM}}:\quad & x_{k + 1} = \frac{1}{k + 2} x_0 + \frac{k + 1}{k + 2} (x_k - \eta\opF_{\cB_k} (x_k))
\end{aligned}
\]
In Experiment~2, we also include \algname{SEG} with averaging, which uses separately tuned parameters.
Both experiments also use the practical single-loop \algname{RAIN} implementation from \citet{chenNearoptimalAlgorithmsMaking2024}:
\begin{align*}
	x_{k + 1 / 2} = x_k - \eta \left(\opF_{\cB_k} (x_k) + r_k (x_k) \right), \qquad x_{k + 1} = x_k - \eta \left(\opF_{\cB_k'} (x_{k + 1 / 2}) + r_k (x_{k + 1 / 2}) \right)
\end{align*}
where
\[
\begin{aligned}
	r_k (v) & = \lambda\gamma\sum_{t = 0}^{k - 1} (1 + \gamma)^t (v - x_t), \qquad \gamma = e^{g / T} - 1, \qquad \lambda = \frac{a}{e^g - 1}, \qquad \eta = \frac{c}{L + a}
\end{aligned}
\]
and $g, a, c$ are tunable hyperparameters.
For all baseline algorithms, we use a fixed oracle budget $Q$ and take $T = \lfloor Q / B\rfloor$ for \algname{SGDA} and \algname{S-OHM}, and $T = \lfloor Q / (2B) \rfloor$ for \algname{SEG} and \algname{RAIN}.

\begin{figure}[!ht]
\centering
\includegraphics[width=\linewidth]{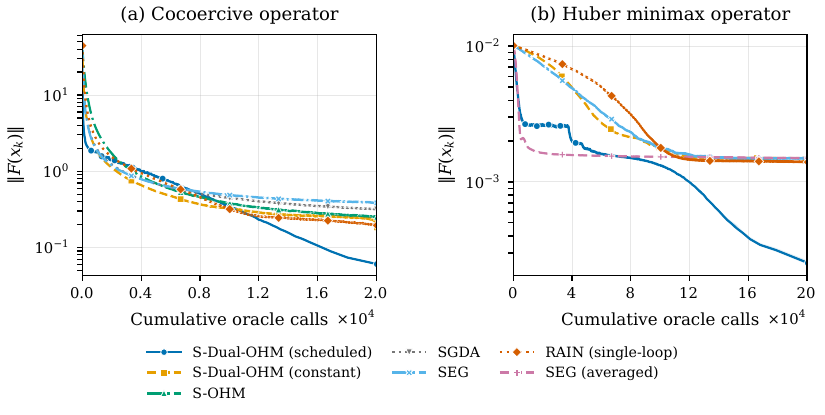}
\vspace{-.5cm}
\caption{Mean residual norms over ten independent runs.
The horizontal axis counts every sampled-operator evaluation, including all samples in mini-batches and inner solves.}
\label{fig:experiments}
\end{figure}

\paragraph{Experiment 1: Cocoercive problem.}
We use the worst-case affine operator construction from \citet{ParkRyu2022_exact}, which has also been considered in \citet{yoonDirectAccelerationStochastic2026}.
Let $d = 1001$, and define
\[
	\opH (x) = \left(x_d - \frac{2}{\sqrt d}, -x_1, \ldots, -x_{d - 1}\right), \qquad \opF (x) = \opH (x - s) + (x - s), \qquad \opT = \opI - \eta \opF
\]
where $s \sim \mathcal N (0, I_d)$ is drawn and fixed across all runs.
The operator $\opF$ is $\nicefrac12$-cocoercive and has a unique zero $x_\star = s + \frac{1}{\sqrt{d}}\boldsymbol 1$.
We use $x_0 = 0$, $\sigma = 2$, and a budget of $2 \times 10^4$ oracle calls.
Our proposed algorithm is the \emph{scheduled} version of \ref{eq:dual}, which uses the allocation profile in~\eqref{eq:optimal-cocoercive}, normalized to the oracle budget.
We provide the full details in Appendix~\ref{app:experiments}.
By contrast, \ref{eq:dual} \emph{(constant)} denotes the version using a constant batch-size $B$.
For that algorithm, \algname{SGDA} and \algname{S-OHM}, we tune $B$ over the grid $\{1, 2, 4, 8, 16, 32, 64, 128, 256\}$, and $\eta$ over $\{0.005, 0.01, 0.025, 0.05, 0.1, 0.25, 0.5, 0.75, 1\}$.
\algname{SEG} uses the same grid for $B$ and $\eta \in \{0.005, 0.01, 0.025, 0.05, 0.1, 0.2, 0.3, 0.45\}$.
\algname{RAIN} uses $B \in \{1, 4, 16, 64\}$, $g \in \{1, 3, 6\}$, $a \in \{0.001, 0.01, 0.1, 1\}$, and $c \in \{0.02, 0.05, 0.1, 0.2, 0.5, 0.8\}$.
We observe that \algname{RAIN} and \ref{eq:dual} (constant) are the most competitive among the baselines, and notably, \ref{eq:dual} (scheduled) outperforms its constant-batch counterpart by a large margin (\cref{fig:experiments}(a)), even though they are sharing the same designing principle and only the batch schedule is chosen differently.

\paragraph{Experiment 2: Beyond cocoercivity.}
We use the Huber minimax problem used in \citet{chenNearoptimalAlgorithmsMaking2024}: writing $c_\nu (t) = \min\{\nu, \max\{-\nu, t\}\}$ componentwise, its saddle operator is
\[
	\opF (x^{(1)}, x^{(2)}) = \left( (1 - \delta) c_\nu \bigl(x^{(1)}\bigr) + \delta x^{(2)}, (1 - \delta) c_\nu \bigl(x^{(2)}\bigr) - \delta x^{(1)}\right)
\]
where $x = (x^{(1)}, x^{(2)}) \in \reals^{1024}$ and we use $\delta = 0.01, \nu = 5 \times 10^{-5}$.
This operator is monotone and $1$-Lipschitz but not cocoercive.
We choose a fixed random unit vector $x_0$, $\sigma = 0.032$, and a budget of $2 \times 10^5$ calls.
Our scheduled \ref{eq:dual} uses $\alpha = \frac{1}{2}$ and the allocation profile in \eqref{eq:weighted-monotone}, normalized to the oracle budget for each candidate $N$.
We provide the full details in Appendix~\ref{app:experiments}.
For \cref{fig:experiments}(b), we use a practical variant of Algorithm~\ref{alg:resolvent} where we use constant-length loop steps ($m_{k, j} \equiv m$), while we also report below the residual value from Algorithm~\ref{alg:resolvent} run with $m_{k, j}$ set precisely as the theory prescribes, which are more conservative at the tested oracle budget.

After an initial joint search over $m$ and $N$, we fix $m = 1$ (two descent iterations per inner sample) and further tune $N$ via refined grid search, separately for the two variants.
We select $N = 24000$ for \ref{eq:dual} (scheduled) and $N = 3200$ for \ref{eq:dual} (constant).
\algname{SEG} selects $B \in \{1, 4, 16, 64, 256\}$ and $\eta \in \{0.005, \allowbreak 0.01, \allowbreak 0.025, \allowbreak 0.05, \allowbreak 0.1, \allowbreak 0.2, \allowbreak 0.4, \allowbreak 0.7, \allowbreak 0.9, \allowbreak 1, \allowbreak 2, \allowbreak 5, \allowbreak 10\}$.
For \algname{RAIN}, we first sweep the grid $(B, g, a, c) \in \{1, 8, 64\} \times \{1, 3, 6\} \times \{0.0001, 0.001, 0.01, 0.1, 1\} \times \{0.02, 0.05, 0.1, 0.2, 0.5, 0.8, 1, 2, 5, 10\}$, and later refine the search using the union of $\{32, 64, 128\} \times \{2, 3, 4, 6\} \times \{0.003, 0.01, 0.03\} \times \{0.35, 0.5, 0.7\}$ and $\{1\} \times \{6\} \times \{0.03, 0.1, 0.3\} \times \{0.005, 0.01, 0.015, 0.02, 0.03, 0.05\}$.

\cref{fig:experiments}(b) shows that the practical variant of scheduled \ref{eq:dual} achieves the smallest residual of $2.55 \times 10^{-4}$.
The final residual values were $1.47 \times 10^{-3}$ for \ref{eq:dual} (constant) and $1.41 \times 10^{-3}$ for \algname{RAIN}.
We also test the theoretically prescribed schedule following Algorithm~\ref{alg:resolvent} and \cref{lem:actual-resolvent}, and observe that it reaches final residual norm $1.46 \times 10^{-3}$ at its tuned horizon $N = 2400$.

\section{Conclusion}
\label{sec:conclusion}

In this work, we show that the robustness of dual-anchor acceleration observed in recent prior work can be translated into near-optimal oracle efficiency for stochastic root-finding.
This is achieved by carefully controlling how errors at different iterations contribute to the residual bound.
In the cocoercive setting, without materially changing the underlying dual-anchor acceleration mechanism but only changing the batch-size schedules, we close the polynomial complexity gap to the lower bound.
We show that the same principle extends to samplewise monotone and Lipschitz problems via stochastic resolvent approximations.
These results strengthen the motivation for studying distinct deterministic acceleration mechanisms; their different responses to noise can provide alternative, clean routes to designing efficient stochastic algorithms.

Our analysis highlights the separation between choosing the acceleration mechanism governing the outer loop and constructing random maps that meet suitable requirements.
While we have used inner loops to handle these requirements, the broader framework we propose is not necessarily confined to algorithms with nested loop structures.
Natural directions for future work include designing a single-loop algorithm for monotone problems from the same principles, removing the logarithmic factors from oracle complexity, developing adaptive budget schedules that require less knowledge of the problem, or extending our framework to horizon-free algorithms.

\label{page:main-end}

\subsection*{AI use statement}
The authors used AI tools to carry out the technical derivations quickly based on our initial conceptual ideas.
We have carefully checked, rederived and significantly restructured the mathematical arguments.
We have also used AI to assist with manuscript drafting and editing, and implementation of the numerical experiments.
The authors checked every part of the paper and take full responsibility for its final content.



\section*{Acknowledgments}
TaeHo Yoon’s contribution to this work was supported by NSF CCF 2504626.
Nicolas Loizou’s contribution to this work was supported by NSF CCF 2504626 and NSF CAREER 2542902.

\bibliographystyle{plainnat}
\bibliography{ref}

\clearpage
\appendix
\part*{Appendix}
\medskip
\tableofcontents
\newpage

\section{Convergence of S-Dual-OHM with general random maps}
\label{app:fixed-point}

We prove Theorem~\ref{thm:outer-fixed-point} and \ref{cor:outer-monotone}.
With $\rho_k = 0$, Theorem~\ref{thm:outer-fixed-point} immediately recovers Theorem~\ref{thm:fixed-point}.
Throughout the section, $\opT$ is nonexpansive and $\opT (x_\star) = x_\star$.

\subsection{Identity for arbitrary evaluation maps}

The following identity holds for~\ref{eq:dual} run with arbitrary maps $\opT_0, \ldots, \opT_{N - 2}$.

\begin{lemma}[\cite{yoonDirectAccelerationStochastic2026}, Lemma~4.2]
\label{lem:app-identity}
Set $a_{N, k} = \frac1{(N - k) (N - k - 1)}$ and
\[
	Q_k = \frac12 \left(\sqnorm{x_k - x_{N - 1}} - \sqnorm{\opT_k (x_k) - \opT (x_{N - 1})}\right).
\]
Then we have
\begin{equation}
\label{eq:app-terminal-identity}
\begin{aligned}
	0 = {} & \frac{N - 1}{4}\sqnorm{x_{N - 1} - \opT (x_{N - 1})} + \frac12\inprod{x_{N - 1} - \opT (x_{N - 1})}{x_{N - 1} - x_0} + \frac N2\sum_{k = 0}^{N - 2}a_{N, k}Q_k .
\end{aligned}
\end{equation}
\end{lemma}

\begin{proof}
Define the potential
\[
\begin{aligned}
	U_k = {} & -\frac{N - k - 1}{N - k} \sqnorm{\opT_{k - 1} (x_{k - 1}) - x_k + x_{N - 1} - \opT (x_{N - 1})} \\
	& + \frac2{N - k} \inprod{\opT_{k - 1} (x_{k - 1}) - x_k + x_{N - 1} - \opT (x_{N - 1})}{x_k - x_{N - 1}}.
\end{aligned}
\]
One can directly verify
\begin{align}
\label{eq:S-Dual-OHM-Lyapunov}
	U_k - U_{k + 1} = 2a_{N, k}Q_k
\end{align}
by substituting the update rule~\ref{eq:dual} and simplifying.
This has been done in an equivalent form using $\opT = \opI - \alpha\opF$ for cocoercive operator $\opF$, in \citet[Lemma~4.2]{yoonDirectAccelerationStochastic2026}.
Because we have $U_{N - 1} = 0$, using $\opT_{-1} (x_{-1}) = x_0$, we have
\[
	U_0 = -\frac{N - 1}{N}\sqnorm{x_{N - 1} - \opT (x_{N - 1})} + \frac2N\inprod{x_{N - 1} - \opT (x_{N - 1})}{x_0 - x_{N - 1}}.
\]
Thus, summing \eqref{eq:S-Dual-OHM-Lyapunov} for $k = 0, \dots, N - 2$ and multiplying by $N / 4$ proves the identity~\eqref{eq:app-terminal-identity}.
\end{proof}

Observe that
\begin{align*}
	& \inprod{x_{N - 1} - \opT (x_{N - 1})}{x_{N - 1} - x_\star} - \frac12\sqnorm{x_{N - 1} - \opT (x_{N - 1})} \\
	& = \frac{1}{2} \left( \sqnorm{x_{N - 1} - x_\star} - \sqnorm{\opT (x_{N - 1}) - x_\star} \right) \ge 0
\end{align*}
Using the above inequality and taking total expectation in \eqref{eq:app-terminal-identity} gives
\begin{align*}
	0 & \ge \expec{\frac{N}{4} \sqnorm{x_{N - 1} - \opT (x_{N - 1})} + \frac{1}{2} \inprod{x_{N - 1} - \opT (x_{N - 1})}{x_\star - x_0}} + \frac{N}{2} \sum_{k = 0}^{N - 2} a_{N, k} \expec{Q_k} \\
	& \ge \frac{N}{8} \expec{\sqnorm{x_{N - 1} - \opT (x_{N - 1})}} - \frac{1}{2N} \sqnorm{x_0 - x_\star} + \frac{N}{2} \sum_{k = 0}^{N - 2} a_{N, k} \expec{Q_k}
\end{align*}
and rearranging, we get
\begin{equation}
\label{eq:app-residual-comparison}
	\expec{\sqnorm{x_{N - 1} - \opT (x_{N - 1})}} \le \frac{4D^2}{N^2} - 4\sum_{k = 0}^{N - 2}a_{N, k}\expec{Q_k} .
\end{equation}

\subsection{Proof of Theorem~\ref{thm:outer-fixed-point}}

Unrolling the \ref{eq:dual} recursion gives
\begin{equation}
\label{eq:app-positive}
	x_j = \frac1N x_0 + \sum_{t = 0}^{j - 2}a_{N, t}\opT_t (x_t) + \frac{N - j}{N - j + 1}\opT_{j - 1} (x_{j - 1})
\end{equation}
for $j = 1, \dots, N - 1$; see \citep[Lemma~4.4]{yoonDirectAccelerationStochastic2026}.
Note that this is a convex combination.
Define the common moment bound in advance: with $R = \sum_{t = 0}^{N - 2}\rho_t$, set
\[
	M^2 = e^R \left(D^2 + \sum_{t = 0}^{N - 2}\sigma_t^2\right) .
\]
Taking $v = x_\star$ in~\eqref{eq:monotone-bias} gives
\[
	\expec{\sqnorm{\opT_k (u) - x_\star}} \le (1 + \rho_k) \sqnorm{u - x_\star} + \sigma_k^2
\]
for $k = 0, 1, \dots, N - 2$.
Therefore, using induction and the fact that $x_j$ is a convex combination of $x_0$ and $\opT_t (x_t)$ for $t = 0, \dots, j - 1$, we obtain
\begin{align}
\label{eq:iterate-moment-bound}
	\expec{\sqnorm{x_j - x_\star}} \le \prod_{t = 0}^{j - 1} (1 + \rho_t) \left(D^2 + \sum_{t = 0}^{j - 1}\sigma_t^2\right) \le M^2.
\end{align}

Now, write $\opS_k = \expec{\opT_k}$ for the deterministic mean map.
By setting $u = v$ in \eqref{eq:monotone-bias} and using the moment bound~\eqref{eq:iterate-moment-bound}, we see that
\begin{equation}
\label{eq:app-centered-variance}
	\nu_k^2 := \expec{\sqnorm{\opT_k (x_k) - \opS_k (x_k)}} \le \expec{\sqnorm{\opT_k (x_k) - \opT (x_k)}} \le \sigma_k^2 + 2\rho_k M^2
\end{equation}
where the first inequality follows from $\expec{\sqnorm{\opT_k (x_k) - \opT (x_k)}} = \expec{\sqnorm{\opT_k (x_k) - \opS_k (x_k)}} + \expec{\sqnorm{\opS_k (x_k) - \opT (x_k)}}$.

Now suppose that we replace only $\opT_k$ by an independent copy $\opT_k'$, and keep all the other maps the same.
Denote the resulting iterates by $x_j^{(k)}$.
The two trajectories agree through time $k$, and the copied trajectory has the same distribution as the original one, so
\[
	x_j^{(k)} = x_j\quad (0 \le j \le k), \qquad \expec{\sqnorm{x_t^{(k)} - x_\star}} \le M^2 \quad (0 \le t \le N - 1).
\]
Conditional on $x_k$, the evaluations $\opT_k (x_k)$ and $\opT_k' (x_k)$ are independent and have the same mean $\opS_k (x_k)$.
Therefore,
\begin{align*}
	\expec{\sqnorm{\opT_k (x_k) - \opT_k' (x_k)}} & = \expec{\sqnorm{\opT_k (x_k) - \opS_k (x_k)}} + \expec{\sqnorm{\opT_k' (x_k) - \opS_k (x_k)}} = 2\nu_k^2.
\end{align*}
For $t > k$, the pair $(x_t, x_t^{(k)})$ is independent of the $\opT_t$, which is shared between the two trajectories.
Thus we can use~\eqref{eq:monotone-stability} and apply \eqref{eq:iterate-moment-bound} to obtain
\begin{align}
	\expec{\sqnorm{\opT_t (x_t) - \opT_t (x_t^{(k)})}} & = \expec{\condexp{\sqnorm{\opT_t (x_t) - \opT_t (x_t^{(k)})}} {x_t, x_t^{(k)}}} \nonumber \\
	& \le \expec{\sqnorm{x_t - x_t^{(k)}}} + \rho_t \left(\expec{\sqnorm{x_t - x_\star}} + \expec{\sqnorm{x_t^{(k)} - x_\star}} + \sigma_t^2\right) \nonumber \\
	& \le \expec{\sqnorm{x_t - x_t^{(k)}}} + \rho_t (2M^2 + \sigma_t^2) . \label{eq:nonexpansive-with-error}
\end{align}
Again, \eqref{eq:app-positive} is a convex combination, and squared norm $\sqnorm{\cdot}$ is convex, so using \eqref{eq:nonexpansive-with-error} we obtain
\begin{equation}
\label{eq:app-replacement-recursion}
\begin{aligned}
	\expec{\sqnorm{x_j - x_j^{(k)}}} & \le 2a_{N, k}\nu_k^2 + \sum_{t = k + 1}^{j - 2}a_{N, t} \left(\expec{\sqnorm{x_t - x_t^{(k)}}} + \rho_t (2M^2 + \sigma_t^2) \right) \\
	& \quad + \frac{N - j}{N - j + 1} \left(\expec{\sqnorm{x_{j - 1} - x_{j - 1}^{(k)}}} + \rho_{j - 1} (2M^2 + \sigma_{j - 1}^2) \right)
\end{aligned}
\end{equation}
for $j > k + 1$.
Using this bound, we will show by induction that, for $j = k + 1, \dots, N - 1$,
\begin{equation}
\label{eq:app-replacement-iterates}
\begin{aligned}
	& \expec{\sqnorm{x_j - x_j^{(k)}}} \\
	& \le \nu_k^2 \left(1 + a_{N, k} (N - j) (N - j - 1) \right) + \frac12\sum_{t = k + 1}^{j - 1}\rho_t (2M^2 + \sigma_t^2) \left(1 + a_{N, t} (N - j) (N - j - 1) \right).
\end{aligned}
\end{equation}
For $j = k + 1$, because $x_k$ and $\opT_{k - 1} (x_{k - 1})$ are shared, we have
\begin{align*}
	\expec{\sqnorm{x_{k + 1} - x_{k + 1}^{(k)}}} & \le \frac{N - k - 1}{N - k} \expec{\sqnorm{\opT_k (x_k) - \opT_k' (x_k)}} \\
	& = \frac{2 (N - k - 1)}{N - k}\nu_k^2 \\
	& = \nu_k^2 \left(1 + a_{N, k} (N - k - 1) (N - k - 2) \right),
\end{align*}
which proves~\eqref{eq:app-replacement-iterates} at the base case.
Now fix $j \ge k + 2$ and suppose the bound holds through $j - 1$.
Substituting it into~\eqref{eq:app-replacement-recursion} and interchanging the finite sums gives
\begin{align}
\label{eqn:xj-xjk-diff-bound-large}
\begin{aligned}
	& \expec{\sqnorm{x_j - x_j^{(k)}}} \\
	& \quad \le \nu_k^2 \Bigg[2a_{N, k} + \sum_{t = k + 1}^{j - 2}a_{N, t} \left(1 + a_{N, k} (N - t) (N - t - 1) \right) \\
	& \hspace{60pt} + \frac{N - j}{N - j + 1} \left(1 + a_{N, k} (N - j + 1) (N - j) \right) \Bigg] \\
	& \qquad + \frac12\sum_{s = k + 1}^{j - 2}\rho_s (2M^2 + \sigma_s^2) \Bigg[2a_{N, s} + \sum_{t = s + 1}^{j - 2}a_{N, t} \left(1 + a_{N, s} (N - t) (N - t - 1) \right) \\
	& \hspace{150pt} + \frac{N - j}{N - j + 1} \left(1 + a_{N, s} (N - j + 1) (N - j) \right) \Bigg] \\
	& \qquad + \frac{N - j}{N - j + 1}\rho_{j - 1} (2M^2 + \sigma_{j - 1}^2).
\end{aligned}
\end{align}
Note that for every $k \le s \le j - 2$, the identity $a_{N, t} = 1 / (N - t - 1) - 1 / (N - t)$ gives
\begin{align*}
	2a_{N, s} + \sum_{t = s + 1}^{j - 2}a_{N, t} + \frac{N - j}{N - j + 1} & = 2a_{N, s} + \frac1{N - j + 1} - \frac1{N - s - 1} + \frac{N - j}{N - j + 1} \\
	& = 1 - \frac1{N - s} + a_{N, s},
\end{align*}
and
\begin{align*}
	& \sum_{t = s + 1}^{j - 2}a_{N, t}a_{N, s} (N - t) (N - t - 1) + \frac{N - j}{N - j + 1}a_{N, s} (N - j + 1) (N - j) \\
	& = a_{N, s} \left( \sum_{t = s + 1}^{j - 2}1 + (N - j)^2 \right) \\
	& = a_{N, s} \left( j - s - 2 + (N - j)^2 \right) \\
	& = a_{N, s} \left( N - s - 2 + (N - j) (N - j - 1) \right).
\end{align*}
Consequently, each square bracket simplifies to
\begin{align*}
	& 2a_{N, s} + \sum_{t = s + 1}^{j - 2}a_{N, t} \left(1 + a_{N, s} (N - t) (N - t - 1) \right) + \frac{N - j}{N - j + 1} \left(1 + a_{N, s} (N - j + 1) (N - j) \right) \\
	& = 1 - \frac1{N - s} + a_{N, s} + a_{N, s} \left(N - s - 2 + (N - j) (N - j - 1) \right) \\
	& = 1 + a_{N, s} (N - j) (N - j - 1)
\end{align*}
and the coefficient of the last term also satisfies
\[
	\frac{N - j}{N - j + 1} = \frac12 \left(1 + a_{N, j - 1} (N - j) (N - j - 1) \right).
\]
Substituting these identities into~\eqref{eqn:xj-xjk-diff-bound-large} proves~\eqref{eq:app-replacement-iterates}, and completes the induction.

In particular, setting $j = N - 1$ in \eqref{eq:app-replacement-iterates} and using $\sigma_t^2 \le M^2$ and $\sum_{t = k + 1}^{N - 2}\rho_t \le R$ gives
\begin{equation}
\label{eq:app-replacement-bound}
\begin{aligned}
	\expec{\sqnorm{x_{N - 1} - x_{N - 1}^{(k)}}} \le \nu_k^2 + \frac12\sum_{t = k + 1}^{N - 2}\rho_t (2M^2 + \sigma_t^2) \le \nu_k^2 + \frac32RM^2.
\end{aligned}
\end{equation}
Now, conditional on the maps before iteration $k$, the error $\opT_k (x_k) - \opS_k (x_k)$ has mean zero and is independent of $x_{N - 1}^{(k)}$.
Thus, Cauchy--Schwarz inequality, nonexpansivity of $\opT$, \eqref{eq:app-centered-variance} and~\eqref{eq:app-replacement-bound} imply
\begin{align}
	\left|\expec{\inprod{\opT (x_{N - 1})} {\opT_k (x_k) - \opS_k (x_k)}}\right| & = \left|\expec{\inprod{\opT (x_{N - 1}) - \opT (x_{N - 1}^{(k)})} {\opT_k (x_k) - \opS_k (x_k)}}\right| \nonumber \\
	& \le \sqrt{\nu_k^2 \left(\nu_k^2 + \frac32RM^2\right)} \le \nu_k^2 + \frac34RM^2. \label{eq:app-fixed-point-covariance}
\end{align}
Now for any $u, v \in \reals^d$, the identity $\expec{\opT_k (u) - \opS_k (u)} = 0$ gives the variance decomposition
\[
	\expec{\sqnorm{\opT_k (u) - \opT (v)}} = \sqnorm{\opS_k (u) - \opT (v)} + \expec{\sqnorm{\opT_k (u) - \opS_k (u)}}.
\]
Subtracting the last term from~\eqref{eq:monotone-bias} therefore yields
\begin{equation}
\label{eq:app-mean-map-pointwise}
\begin{aligned}
	\sqnorm{\opS_k (u) - \opT (v)} \le {} & \sqnorm{u - v} + \rho_k \bigl(\sqnorm{u - x_\star} + \sqnorm{v - x_\star}\bigr) + \sigma_k^2 - \expec{\sqnorm{\opT_k (u) - \opS_k (u)}}.
\end{aligned}
\end{equation}
Here the expectation is taken over $\opT_k$ with $u$ fixed, so both sides are deterministic functions of $(u, v)$, and in particular~\eqref{eq:app-mean-map-pointwise} holds pathwise with $(u, v) = (x_k, x_{N - 1})$.
Thus, almost surely,
\begin{equation}
\label{eq:app-mean-map-iterates}
\begin{aligned}
	& \sqnorm{\opS_k (x_k) - \opT (x_{N - 1})} \\
	& \le \sqnorm{x_k - x_{N - 1}} + \sigma_k^2 + \rho_k \bigl(\sqnorm{x_k - x_\star} + \sqnorm{x_{N - 1} - x_\star}\bigr) - \condexp{\sqnorm{\opT_k (x_k) - \opS_k (x_k)}}{x_k}.
\end{aligned}
\end{equation}
Now we manipulate the quantities $Q_k$ appearing in \eqref{eq:app-residual-comparison} as follows:
\begin{align*}
	Q_k & = \frac12 \left(\sqnorm{x_k - x_{N - 1}} - \sqnorm{\opT_k (x_k) - \opT (x_{N - 1})}\right) \\
	& = \frac12 \left(\sqnorm{x_k - x_{N - 1}} - \sqnorm{\opS_k (x_k) - \opT (x_{N - 1})}\right) - \frac12\sqnorm{\opT_k (x_k) - \opS_k (x_k)} \\
	& \quad - \inprod{\opS_k (x_k)}{\opT_k (x_k) - \opS_k (x_k)} + \inprod{\opT (x_{N - 1})}{\opT_k (x_k) - \opS_k (x_k)}.
\end{align*}
By independence of $x_k$ and $\opT_k$, we have $\condexp{\opT_k (x_k) - \opS_k (x_k)}{x_k} = 0$, which implies
\begin{align*}
	& \expec{\inprod{\opS_k (x_k)}{\opT_k (x_k) - \opS_k (x_k)}} = \expec{\inprod{\opS_k (x_k)} {\condexp{\opT_k (x_k) - \opS_k (x_k)}{x_k}}} = 0 \\
	& \expec{\condexp{\sqnorm{\opT_k (x_k) - \opS_k (x_k)}}{x_k}} = \expec{\sqnorm{\opT_k (x_k) - \opS_k (x_k)}} = \nu_k^2.
\end{align*}
Taking expectations in the expansion of $Q_k$ and applying~\eqref{eq:app-mean-map-iterates} now shows explicitly how the variance terms cancel:
\begin{align*}
	\expec{Q_k} & \ge - \frac{\sigma_k^2}{2} - \frac{\rho_k}{2} \left(\expec{\sqnorm{x_k - x_\star}} + \expec{\sqnorm{x_{N - 1} - x_\star}}\right) + \frac{\nu_k^2}{2} - \frac{\nu_k^2}{2} \\
	& \quad + \expec{\inprod{\opT (x_{N - 1})} {\opT_k (x_k) - \opS_k (x_k)}} \\
	& \ge - \frac{\sigma_k^2}{2} - \rho_kM^2 + \expec{\inprod{\opT (x_{N - 1})} {\opT_k (x_k) - \opS_k (x_k)}},
\end{align*}
where the last inequality uses~\eqref{eq:iterate-moment-bound}.
Finally,~\eqref{eq:app-fixed-point-covariance} and~\eqref{eq:app-centered-variance} yield
\begin{align*}
	\expec{Q_k} & \ge - \frac{\sigma_k^2}{2} - \rho_kM^2 - \nu_k^2 - \frac34RM^2 \\
	& \ge - \frac{\sigma_k^2}{2} - \rho_kM^2 - \bigl(\sigma_k^2 + 2\rho_kM^2\bigr) - \frac34RM^2 \\
	& = -\frac32\sigma_k^2 - 3\rho_kM^2 - \frac34RM^2.
\end{align*}
Using this estimate in~\eqref{eq:app-residual-comparison} gives
\begin{align*}
	\expec{\sqnorm{x_{N - 1} - \opT (x_{N - 1})}} & \le \frac{4D^2}{N^2} + 6\sum_{k = 0}^{N - 2}a_{N, k}\sigma_k^2 \\
	& \quad + 12M^2\sum_{k = 0}^{N - 2}a_{N, k}\rho_k + 3RM^2\sum_{k = 0}^{N - 2}a_{N, k}.
\end{align*}
Since $a_{N, k} \le 1$ and $\sum_k a_{N, k} = 1 - 1 / N < 1$, the last two terms are bounded by $15RM^2$, proving Theorem~\ref{thm:outer-fixed-point}.

\subsection{Proof of Corollary~\ref{cor:outer-monotone}}

For this result, we evaluate one more random map $\opT_{N - 1}$ because $y_N = \opT_{N - 1} (x_{N - 1})$ is the output of the algorithm.
Therefore, we let $R = \sum_{k = 0}^{N - 1}\rho_k$ and $M^2 = e^R \left(D^2 + \sum_{k = 0}^{N - 1}\sigma_k^2 \right)$.
The same moment bound~\eqref{eq:iterate-moment-bound} still applies to $x_{N - 1}$.
Now by~\eqref{eq:monotone-bias} we have
\[
	\expec{\sqnorm{\opT_{N - 1} (x_{N - 1}) - \opT (x_{N - 1})}} \le 2\rho_{N - 1}M^2 + \sigma_{N - 1}^2.
\]
Observe that because $\opT = \opJ_{\alpha \opF}$, we have $x_{N - 1} - \opT (x_{N - 1}) = \alpha\opF (\opT (x_{N - 1}))$, so
\begin{align*}
	\sqnorm{\opF (y_N)} & \le 2\sqnorm{\opF (y_N) - \opF (\opT (x_{N - 1}))} + 2\sqnorm{\opF (\opT (x_{N - 1}))} \\
	& \le 2L^2 \sqnorm{y_N - \opT (x_{N - 1})} + \frac{2}{\alpha^2} \sqnorm{x_{N - 1} - \opT (x_{N - 1})} \\
	& = 2L^2 \sqnorm{\opT_{N - 1} (x_{N - 1}) - \opT (x_{N - 1})} + \frac{2}{\alpha^2} \sqnorm{x_{N - 1} - \opT (x_{N - 1})} .
\end{align*}
Taking expectation and applying the bounds from \cref{thm:outer-fixed-point}, we obtain
\begin{align*}
	\expec{\sqnorm{\opF (y_N)}} & \le \frac2{\alpha^2}\expec{\sqnorm{x_{N - 1} - \opT (x_{N - 1})}} + 2L^2\expec{\sqnorm{\opT_{N - 1} (x_{N - 1}) - \opT (x_{N - 1})}} \\
	& \le \frac{8D^2}{\alpha^2N^2} + \frac{12}{\alpha^2}\sum_{k = 0}^{N - 2}a_{N, k}\sigma_k^2 + 2L^2\sigma_{N - 1}^2 + M^2 \left(\frac{30}{\alpha^2}\sum_{k = 0}^{N - 2}\rho_k + 4L^2\rho_{N - 1}\right).
\end{align*}
Because $\alpha L \le 1$ and $a_{N, k} \le \frac{2}{(N - k)^2}$, the last line is upper-bounded by~\eqref{eq:abstract-outer-bound}, completing the proof.

\clearpage
\section{Analysis of approximate resolvents for monotone problems}
\label{app:monotone}

The analysis below verifies Condition~\ref{ass:monotone-Tk} for the finite maps computed by Algorithm~\ref{alg:resolvent}.
Exact sampled resolvents appear only in the proof of these map properties.

\subsection{Basic handy lemmas}

\begin{lemma}
\label{lem:simple-resolvent-bound}
Let $\opG \colon \reals^d \to \reals^d$ be single-valued and maximally monotone.
For any $x \in \reals^d$, we have $\norm{x - \opJ_\opG (x)} \le \norm{\opG (x)}$.
\end{lemma}

\begin{proof}
Note that $x = \opJ_\opG (x + \opG (x))$, so
\begin{align*}
	\norm{x - \opJ_\opG (x)} = \norm{\opJ_\opG (x + \opG (x)) - \opJ_\opG (x)} \le \norm{(x + \opG (x)) - x} = \norm{\opG (x)} .
\end{align*}

\end{proof}

\begin{lemma}
\label{lem:resolvent}
Under Assumptions~\ref{ass:oracle} and~\ref{ass:samplewise}, for $\beta > 0$ and any deterministic $u, v \in \reals^d$,
\begin{align}
\label{eq:resolvent-comparison}
	\expec{\sqnorm{\opJ_{\beta\opF (\cdot; \xi)} (u) - \opJ_{\beta\opF} (v)}} \le \sqnorm{u - v} + \beta^2\sigma^2.
\end{align}
\end{lemma}

\begin{proof}
Let $y = \opJ_{\beta\opF} (v)$ and $\delta_\xi = \opF (y; \xi) - \opF (y)$.
Then we have $v = y + \beta\opF (y) = y + \beta\opF (y; \xi) - \beta\delta_\xi$.
Therefore,
\begin{align*}
	y + \beta\opF (y; \xi) = v + \beta\delta_\xi \iff \opJ_{\beta\opF} (v) = y = \opJ_{\beta\opF (\cdot; \xi)} (v + \beta\delta_\xi) .
\end{align*}
This implies, for any $u, v \in \reals^d$,
\begin{align*}
	\expec{\sqnorm{\opJ_{\beta\opF (\cdot; \xi) (u)} - \opJ_{\beta\opF} (v)}} & = \expec{\sqnorm{\opJ_{\beta\opF (\cdot; \xi) (u)} - \opJ_{\beta\opF (\cdot; \xi)} (v + \beta\delta_\xi)} } \\
	& \le \expec{\sqnorm{u - v - \beta\delta_\xi} } \\
	& \le \sqnorm{u - v} + \beta^2 \sigma^2
\end{align*}
where the last line uses $\expec{\delta_\xi} = 0$ and $\expec{\sqnorm{\delta_\xi}} \le \sigma^2$.

\end{proof}

\subsection{Proof of Lemma~\ref{lem:actual-resolvent}}
\label{app:finite-streaming}

Let us denote the ideal iterates from \eqref{eq:streaming-ideal} by $z_j (u)$, so that $z_{B_k} (u) = \widehat\opT_k (u)$.
Because $\opT = \opJ_{\alpha\opF}$, for any $v \in \reals^d$, we have $v = \opT (v) + \alpha\opF (\opT (v))$.
Multiplying $\gamma_j$ throughout and rearranging, we obtain
\[
	\opT (v) + \alpha\gamma_j \opF (\opT (v)) = (1 - \gamma_j) \opT (v) + \gamma_j v \iff \opT (v) = \opJ_{\alpha\gamma_j\opF} ( (1 - \gamma_j) \opT (v) + \gamma_jv).
\]
Therefore, applying~\eqref{eq:resolvent-comparison} with $\beta = \alpha\gamma_j$, we obtain
\begin{align*}
	& \expec{\sqnorm{z_{j + 1} (u) - \opT (v)}} \\
	& = \expec{\sqnorm{\opJ_{\alpha\gamma_j \opF (\cdot; \xi_j)} \left( (1 - \gamma_j) z_j (u) + \gamma_j u\right) - \opJ_{\alpha\gamma_j\opF} \left( (1 - \gamma_j) \opT (v) + \gamma_j v \right)} } \\
	& \le \sqnorm{(1 - \gamma_j) z_j (u) + \gamma_j u - \left( (1 - \gamma_j) \opT (v) + \gamma_j v \right)} + \alpha^2 \gamma_j^2 \sigma^2 \\
	& = \sqnorm{(1 - \gamma_j) (z_j (u) - \opT (v)) + \gamma_j (u - v)} + \alpha^2 \gamma_j^2 \sigma^2 \\
	& \le (1 - \gamma_j) \sqnorm{ z_j (u) - \opT (v) } + \gamma_j \sqnorm{u - v} + \alpha^2 \gamma_j^2 \sigma^2
\end{align*}
where the expectation is conditioned on the randomness revealed before $\xi_j$, and the last line uses convexity of $\sqnorm{\cdot}$.
Now subtracting $\sqnorm{u - v}$ from both sides of the above, using $\gamma_j = \frac{2}{j + 2}$ and taking total expectation, we obtain
\[
\begin{aligned}
	\expec{\sqnorm{z_{j + 1} (u) - \opT (v)}} - \sqnorm{u - v} \le \frac{j}{j + 2} \left(\expec{\sqnorm{z_j (u) - \opT (v)}} - \sqnorm{u - v}\right) + \frac{4\alpha^2\sigma^2}{(j + 2)^2}.
\end{aligned}
\]
Multiplying by $(j + 1) (j + 2)$ and chaining the resulting inequality for $j = 0, \dots, B_k - 1$ we obtain
\[
\begin{aligned}
	B_k (B_k + 1) \left(\expec{\sqnorm{z_{B_k} (u) - \opT (v)}} - \sqnorm{u - v}\right) \le 4 \alpha^2\sigma^2\sum_{j = 0}^{B_k - 1}\frac{j + 1}{j + 2} \le 4\alpha^2\sigma^2 B_k.
\end{aligned}
\]
Dividing both sides by $B_k (B_k + 1)$ and using $z_{B_k} (u) = \widehat\opT (u)$, we obtain~\eqref{eq:streaming-comparison}.

Fix $k$ and $u \in \reals^d$.
Consider the iterates $w_j$ from Algorithm~\ref{alg:resolvent} and the ideal iterates $z_j = z_j (u)$ from \eqref{eq:streaming-ideal}, obtained by using a same sequence of samples.
Below, we drop the $u$-dependence in $z_j$ and $w_j$.
Recall that $p_j = (1 - \gamma_j) w_j + \gamma_ju$ and $w_{j + 1}$ is obtained by the following iteration: $q_0 = p_j$,
\[
	q_{r + 1} = p_j - \alpha\gamma_j \opF (q_r; \xi_j)
\]
for $j = 0, \dots, m_{k, j}$, which is the Picard iteration with respect to the operator $x \mapsto p_j - \alpha\gamma_j \opF (x; \xi_j)$, which has the contraction factor $\frac{1}{j + 2}$ because
\begin{align*}
	& \norm{\left(p_j - \alpha\gamma_j \opF (x; \xi_j) \right) - \left(p_j - \alpha\gamma_j \opF (y; \xi_j) \right)} \\
	& \quad \le \alpha\gamma_j \norm{\opF (x; \xi_j) - \opF (y; \xi_j)} \le \alpha L \gamma_j \norm{x - y} = \frac{1}{j + 2} \norm{x - y}
\end{align*}
for any $x, y \in \reals^d$.
Therefore,
\begin{align}
	\norm{w_{j + 1} - \opJ_{\alpha\gamma_j\opF (\cdot; \xi_j)} (p_j)} & \le \left(\frac{1}{j + 2}\right)^{m_{k, j} + 1} \norm{p_j - \opJ_{\alpha\gamma_j\opF (\cdot; \xi_j)} (p_j)} \nonumber \\
	& \le \frac{\rho_k}{16 (j + 2)} \norm{p_j - \opJ_{\alpha\gamma_j\opF (\cdot; \xi_j)} (p_j)} \nonumber \\
	& \le \frac{\rho_k\alpha\gamma_j}{16 (j + 2)}\norm{\opF (p_j; \xi_j)} \label{eq:wj-pj-resolvent-diff}
\end{align}
where the second line uses $m_{k, j} = \left\lceil \frac{\log (16 / \rho_k)}{\log (j + 2)} \right\rceil$ and the last line uses \cref{lem:simple-resolvent-bound} with $\opG = \alpha\gamma_j \opF (\cdot; \xi_j)$.
On the other hand, we have
\begin{align}
	\norm{\opJ_{\alpha\gamma_j\opF (\cdot; \xi_j)} (p_j) - x_\star} & \le \norm{\opJ_{\alpha\gamma_j\opF (\cdot; \xi_j)} (p_j) - \opJ_{\alpha\gamma_j\opF (\cdot; \xi_j)} (x_\star)} + \norm{\opJ_{\alpha\gamma_j\opF (\cdot; \xi_j)} (x_\star) - x_\star} \nonumber \\
	& \le \norm{p_j - x_\star} + \alpha\gamma_j \norm{\opF (x_\star; \xi_j)} \label{eq:pj-resolvent-star-diff}
\end{align}
where the last line again uses \cref{lem:simple-resolvent-bound} with $\opG = \alpha\gamma_j \opF (\cdot; \xi_j)$, and
\begin{align}
\label{eq:Fpj-norm-bound}
	\norm{\opF (p_j ; \xi_j)} \le \norm{\opF (p_j ; \xi_j) - \opF (x_\star; \xi_j)} + \norm{\opF (x_\star; \xi_j)} \le L\norm{p_j - x_\star} + \norm{\opF (x_\star; \xi_j)} .
\end{align}

Now for any random quantities, consider their $L^2$-norms $\norm{X}_2 := \expec{\sqnorm{X}}^{1 / 2}$.
Using Minkowski inequality and \eqref{eq:wj-pj-resolvent-diff}--\eqref{eq:Fpj-norm-bound}, we have
\begin{align*}
	\norm{w_{j + 1} - x_\star}_2 & \le \norm{w_{j + 1} - \opJ_{\alpha\gamma_j\opF (\cdot; \xi_j)} (p_j)}_2 + \norm{\opJ_{\alpha\gamma_j\opF (\cdot; \xi_j)} (p_j) - x_\star}_2 \\
	& \le \frac{\rho_k\alpha\gamma_j}{16 (j + 2)}\norm{\opF (p_j; \xi_j)}_2 + \norm{p_j - x_\star}_2 + \alpha\gamma_j \norm{\opF (x_\star; \xi_j)}_2 \\
	& \le \frac{\rho_k\alpha\gamma_j}{16 (j + 2)} \left( L\norm{p_j - x_\star}_2 + \norm{\opF (x_\star; \xi_j)}_2 \right) + \norm{p_j - x_\star}_2 + \alpha\gamma_j \sigma \\
	& \le \left( 1 + \frac{\rho_k \gamma_j}{32} \right) \norm{p_j - x_\star}_2 + \left(1 + \frac{\rho_k}{16}\right) \alpha\gamma_j\sigma \\
	& \le \left( 1 + \frac{\rho_k \gamma_j}{32} \right) \left( (1 - \gamma_j) \norm{w_j - x_\star}_2 + \gamma_j \norm{u - x_\star}_2 \right) + \left(1 + \frac{\rho_k}{16}\right) \alpha\gamma_j\sigma \\
	& \le \left( 1 - \frac{3\gamma_j}{4} \right) \norm{w_j - x_\star}_2 + \frac{5\gamma_j}{4} \norm{u - x_\star}_2 + \frac{3\alpha\gamma_j \sigma}{2}
\end{align*}
where we use $\frac{\alpha L}{j + 2} = \frac{1}{2 (j + 2)} \le \frac{1}{2}$,
\begin{align*}
	\norm{\opF (x_\star; \xi_j)}_2 = \expec{\sqnorm{\opF (x_\star; \xi_j) - \opF (x_\star)}}^{1 / 2} \le \sigma ,
\end{align*}
and $\rho_k \le 1$.
Noting that $p_{j + 1}$ is a convex combination of $w_{j + 1}$ and $u$, we can use induction starting from $w_0 = u$ to obtain
\begin{align}
\label{eq:wj-pj-dist-max}
	\max \left\{ \norm{w_j - x_\star}_2 , \norm{p_j - x_\star}_2 \right\} \le 2 \left( \norm{u - x_\star} + \alpha\sigma \right) .
\end{align}
Now
\begin{align*}
	& \norm{w_{j + 1} - z_{j + 1}}_2 \\
	& \le \norm{w_{j + 1} - \opJ_{\alpha\gamma_j\opF (\cdot; \xi_j)} (p_j)}_2 + \norm{\opJ_{\alpha\gamma_j\opF (\cdot; \xi_j)} (p_j) - z_{j + 1}}_2 \\
	& \le \norm{w_{j + 1} - \opJ_{\alpha\gamma_j\opF (\cdot; \xi_j)} (p_j)}_2 + \norm{\opJ_{\alpha\gamma_j\opF (\cdot; \xi_j)} \left( (1 - \gamma_j) w_j + \gamma_j u \right) - \opJ_{\alpha\gamma_j\opF (\cdot; \xi_j)} \left( (1 - \gamma_j) z_j + \gamma_j u \right)}_2 \\
	& \le \frac{\rho_k\alpha\gamma_j}{16 (j + 2)}\norm{\opF (p_j; \xi_j)}_2 + \norm{(1 - \gamma_j) (w_j - z_j)}_2 \\
	& \le (1 - \gamma_j) \norm{w_j - z_j}_2 + \frac{\rho_k\gamma_j}{16 (j + 2)} (\norm{u - x_\star} + 2\alpha\sigma) ,
\end{align*}
where the third inequality uses \eqref{eq:wj-pj-resolvent-diff}, and the last inequality combines \eqref{eq:Fpj-norm-bound} with \eqref{eq:wj-pj-dist-max}.
Chaining this result for $j = 0, \dots, B_k - 1$ and using $w_0 = z_0$ and $\prod_{s = j + 1}^{B_k - 1} (1 - \gamma_s) = \frac{(j + 1) (j + 2)}{B_k (B_k + 1)}$, we obtain
\begin{align}
	\sqrt{\expec{\sqnorm{\opT_k (u) - \widehat\opT_k (u)}}} & = \norm{w_{B_k} - z_{B_k}}_2 \nonumber \\
	& \le \sum_{j = 0}^{B_k - 1}\frac{j + 1}{j + 2} \frac{\rho_k}{8B_k (B_k + 1)} (\norm{u - x_\star} + 2\alpha\sigma) \nonumber \\
	& \le \frac{\rho_k}{8 (B_k + 1)} (\norm{u - x_\star} + 2\alpha\sigma). \label{eq:app-local-streaming}
\end{align}
Now we use Young's inequality with \eqref{eq:streaming-comparison} and~\eqref{eq:app-local-streaming}:
\begin{align*}
	\expec{\sqnorm{\opT_k (u) - \opT (v)}} & \le \left(1 + \frac{\rho_k}{8}\right) \expec{\sqnorm{\widehat\opT_k (u) - \opT (v)}} + \left(1 + \frac{8}{\rho_k}\right) \expec{\sqnorm{\opT_k (u) - \widehat\opT_k (u)}} \\
	& \le \left(1 + \frac{\rho_k}{8}\right) \left(\sqnorm{u - v} + \frac{4\alpha^2\sigma^2}{B_k + 1}\right) + \frac{\rho_k^2 \left(1 + \frac{8}{\rho_k}\right)}{64 (B_k + 1)^2} (\norm{u - x_\star} + 2\alpha\sigma)^2 \\
	& \le \left(1 + \frac{\rho_k}{8}\right) \sqnorm{u - v} + \frac{\rho_k \left(1 + \frac{\rho_k}{8}\right)}{4 (B_k + 1)^2} \sqnorm{u - x_\star} \\
	& \quad + \left(1 + \frac{\rho_k}{8}\right) \left( \frac{4}{B_k + 1} + \frac{\rho_k}{(B_k + 1)^2} \right) \alpha^2 \sigma^2 \\
	& \le \sqnorm{u - v} + \frac{\rho_k}{4} \left( 1 + \frac{\left(1 + \frac{\rho_k}{8}\right)}{(B_k + 1)^2} \right) \sqnorm{u - x_\star} + \frac{\rho_k}{4} \sqnorm{v - x_\star} \\
	& \quad + \left(1 + \frac{\rho_k}{8}\right) \frac{4 + \frac{\rho_k}{2}}{B_k + 1} \alpha^2 \sigma^2 \\
	& \le \sqnorm{u - v} + \rho_k \left( \sqnorm{u - x_\star} + \sqnorm{v - x_\star} \right) + \frac{6\alpha^2 \sigma^2}{B_k + 1}
\end{align*}
where the third inequality uses $\frac{\rho_k^2 \left(1 + \frac{8}{\rho_k}\right)}{64 (B_k + 1)^2} = \frac{\rho_k \left(1 + \frac{\rho_k}{8}\right)}{8 (B_k + 1)^2}$ and $(\norm{u - x_\star} + 2\alpha\sigma)^2 \le 2\sqnorm{u - x_\star} + 8\alpha^2\sigma^2$, the fourth inequality uses $\frac{\rho_k}{8}\sqnorm{u - v} \le \frac{\rho_k}{4}\sqnorm{u - x_\star} + \frac{\rho_k}{4}\sqnorm{v - x_\star}$, and the last inequality follows from $B_k \ge 1$ and $0 < \rho_k \le 1$, which proves \eqref{eq:monotone-bias} with $\sigma_k^2 = \frac{6\alpha^2 \sigma^2}{B_k + 1}$.

Now we turn to showing \eqref{eq:monotone-stability}.
First, since $z_{j + 1} (u) = \opJ_{\alpha\gamma_j\opF (\cdot; \xi_j)} \bigl( (1 - \gamma_j) z_j (u) + \gamma_j u\bigr)$, we have
\[
	\norm{z_{j + 1} (u) - z_{j + 1} (v)} \le (1 - \gamma_j) \norm{z_j (u) - z_j (v)} + \gamma_j\norm{u - v}
\]
for any $u, v \in \reals^d$ by nonexpansivity of $\opJ_{\alpha\gamma_j\opF (\cdot; \xi_j)}$.
Induction starting from $z_0 (u) = u$ and $z_0 (v) = v$ shows that $\widehat\opT_k$ is pathwise nonexpansive, i.e.,
\begin{align*}
	\norm{\widehat\opT_k (u) - \widehat\opT_k (v)} = \norm{z_{B_k} (u) - z_{B_k} (v)} \le \norm{u - v} .
\end{align*}
We use Young's inequality multiple times with \eqref{eq:app-local-streaming}:
\begin{align*}
	& \expec{\sqnorm{\opT_k (u) - \opT_k (v)}} \\
	& \le \left(1 + \frac{\rho_k}{8}\right) \expec{\sqnorm{\widehat\opT_k (u) - \widehat\opT_k (v)}} + \left(1 + \frac{8}{\rho_k}\right) \expec{\sqnorm{\left( \opT_k (u) - \widehat\opT_k (u) \right) - \left( \opT_k (v) - \widehat\opT_k (v) \right)}} \\
	& \le \left(1 + \frac{\rho_k}{8}\right) \sqnorm{u - v} + 2 \left(1 + \frac{8}{\rho_k}\right) \left( \expec{\sqnorm{\opT_k (u) - \widehat\opT_k (u)}} + \expec{\sqnorm{\opT_k (v) - \widehat\opT_k (v)}} \right) \\
	& \le \left(1 + \frac{\rho_k}{8}\right) \sqnorm{u - v} + \frac{2\rho_k^2 \left(1 + \frac{8}{\rho_k}\right)}{64 (B_k + 1)^2} \left( (\norm{u - x_\star} + 2\alpha\sigma)^2 + (\norm{v - x_\star} + 2\alpha\sigma)^2 \right) \\
	& \le \left(1 + \frac{\rho_k}{8}\right) \sqnorm{u - v} + \frac{\rho_k \left(1 + \frac{\rho_k}{8}\right)}{2 (B_k + 1)^2} \left( \sqnorm{u - x_\star} + \sqnorm{v - x_\star} + 8\alpha^2 \sigma^2 \right) \\
	& \le \sqnorm{u - v} + \frac{\rho_k}{4} \left( 1 + \frac{2 \left(1 + \frac{\rho_k}{8}\right)}{(B_k + 1)^2} \right) \left( \sqnorm{u - x_\star} + \sqnorm{v - x_\star} \right) + \frac{4\rho_k \left(1 + \frac{\rho_k}{8}\right)}{(B_k + 1)^2} \alpha^2 \sigma^2 \\
	& \le \sqnorm{u - v} + \rho_k \left( \sqnorm{u - x_\star} + \sqnorm{v - x_\star} \right) + \frac{6\rho_k \alpha^2 \sigma^2}{B_k + 1}
\end{align*}
where the last line again uses $B_k \ge 1$ and $0 < \rho_k \le 1$.

Finally, it remains to bound the number of oracle evaluations at iteration $k$, which is
\[
	\sum_{j = 0}^{B_k - 1} (m_{k, j} + 1) = \sum_{j = 0}^{B_k - 1} \left(1 + \left\lceil\frac{\log (16 / \rho_k)}{\log (j + 2)}\right\rceil \right) \le 2B_k + \log \frac{16}{\rho_k} \sum_{j = 0}^{B_k - 1}\frac1{\log (j + 2)} .
\]
We can bound the final summation using $\sum_{j = 0}^{B_k - 1} \frac{1}{\log (j + 2)} \le \frac{3 (B_k + 1)}{\log (B_k + 1)}$, which completes the proof.

\subsection{Proof of Corollary~\ref{cor:weighted}}
\label{app:allocation-streaming}

First fix $N \ge 2$ and $\rho_k = 1 / (2N^3)$, and consider arbitrary deterministic integer lengths $B_k \ge 1$.
Since $\log \frac{16}{\rho_k} = \log (32N^3) \le 4\log (eN)$, Lemma~\ref{lem:actual-resolvent} bounds the total oracle cost by
\[
	\cO\! \left(\sum_{k = 0}^{N - 1}B_k + \log (eN) \sum_{k = 0}^{N - 1}\frac{B_k}{\log (B_k + 1)}\right).
\]
For any $b, t \ge 1$,
\[
	\frac b{\log (b + 1)} \le \frac{2b}{t} + \frac{e^{t / 2}}{\log2}.
\]
Indeed, the first term suffices when $b \ge e^{t / 2}$, and the second suffices otherwise.
Applying this with $b = B_k$ and $t = \log \frac{eN}{N - k}$, and using $\sum_{k = 0}^{N - 1} (N - k)^{-1 / 2} \le 2\sqrt N$, we get
\[
	\sum_{k = 0}^{N - 1}\frac{B_k}{\log (B_k + 1)} \le 2\sum_{k = 0}^{N - 1}\frac{B_k}{\log (eN / (N - k))} + \frac{2\sqrt e}{\log2}N.
\]
Also, $\log \frac{eN}{N - k} \le \log (eN)$ implies
\[
	\sum_{k = 0}^{N - 1}B_k \le \log (eN) \sum_{k = 0}^{N - 1}\frac{B_k}{\log (eN / (N - k))}.
\]
Consequently, for any choice of batch-sizes $B_k \ge 1$, the total oracle cost is bounded by
\begin{equation}
\label{eq:app-general-allocation-cost}
	\cO\! \left(\log (eN) \left[N + \sum_{k = 0}^{N - 1} \frac{B_k}{\log (eN / (N - k))}\right] \right).
\end{equation}

We now specialize to $B_k$ in~\eqref{eq:weighted-monotone} and $N$ in Corollary~\ref{cor:weighted}.
Note that $R = \sum_{k = 0}^{N - 1} \frac{1}{2N^3} = \frac{1}{2N^2} \le \frac{1}{8}$.
We use Corollary~\ref{cor:outer-monotone} and Lemma~\ref{lem:actual-resolvent} with $\alpha = \frac{1}{2L}$, and use $e^{1 / 8} < \frac{8}{7}$ and
\[
	R\sum_{k = 0}^{N - 1}\sigma_k^2 = \frac1{2N^2}\sum_{k = 0}^{N - 1}\sigma_k^2 \le \sum_{k = 0}^{N - 1}\frac{\sigma_k^2}{2 (N - k)^2},
\]
to obtain
\begin{align*}
	\expec{\sqnorm{\opF (y_N)}} & \le \frac{704}{7}\frac{L^2D^2}{N^2} + \frac{1728}{7}\sigma^2 \sum_{k = 0}^{N - 1}\frac1{(N - k)^2 (B_k + 1)} \\
	& \le 256 \left(\frac{L^2D^2}{N^2} + \sigma^2\sum_{k = 0}^{N - 1}\frac1{(N - k)^2 (B_k + 1)}\right).
\end{align*}
Because $x \mapsto 1 / (x\sqrt{\log (eN / x)})$ is a decreasing function on $[1, N]$,
\begin{equation}
\label{eq:app-sqrtlog-sum}
\begin{aligned}
	\sum_{k = 0}^{N - 1}\frac1{(N - k) \sqrt{\log (eN / (N - k))}} & \le \frac1{\sqrt{\log (eN)}} + \int_1^N\frac{dx}{x\sqrt{\log (eN / x)}} \\
	& = \frac1{\sqrt{\log (eN)}} + 2 \bigl(\sqrt{\log (eN)} - 1\bigr) \le 2\sqrt{\log (eN)}.
\end{aligned}
\end{equation}
The batch allocation~\eqref{eq:weighted-monotone} therefore satisfies
\begin{equation}
\label{eq:app-allocation-variance}
\begin{aligned}
	\sigma^2\sum_{k = 0}^{N - 1}\frac1{(N - k)^2 (B_k + 1)} & \le \frac{\epsilon^2}{1024\sqrt{\log (eN)}} \sum_{k = 0}^{N - 1}\frac1{(N - k) \sqrt{\log (eN / (N - k))}} \le \frac{\epsilon^2}{512}.
\end{aligned}
\end{equation}
Since $N \ge \frac{32LD}{\epsilon}$, we conclude
\[
	\expec{\sqnorm{\opF (y_N)}} \le \frac{256L^2D^2}{N^2} + 256\sigma^2\sum_{k = 0}^{N - 1}\frac1{(N - k)^2 (B_k + 1)} \le \frac{\epsilon^2}{4} + \frac{\epsilon^2}{2} \le \epsilon^2.
\]
To bound the complexity, observe that
\[
	B_k \le 1 + \frac{1024\sigma^2\sqrt{\log (eN)}}{\epsilon^2} \frac{\sqrt{\log (eN / (N - k))}}{N - k}.
\]
Since $\log \frac{eN}{N - k} \ge 1$, we obtain from~\eqref{eq:app-sqrtlog-sum}:
\begin{align*}
	\sum_{k = 0}^{N - 1}\frac{B_k}{\log (eN / (N - k))} & \le N + \frac{1024\sigma^2\sqrt{\log (eN)}}{\epsilon^2} \sum_{k = 0}^{N - 1}\frac1{(N - k) \sqrt{\log (eN / (N - k))}} \\
	& \le N + \frac{2048\sigma^2}{\epsilon^2}\log (eN).
\end{align*}
Substituting this into~\eqref{eq:app-general-allocation-cost} therefore bounds the total number of oracle evaluations by
\[
	\cO\! \left(N\log (eN) + \frac{\sigma^2}{\epsilon^2}\log^2 (eN) \right).
\]
Since $N = \cO \left( 1 + \frac{LD}{\epsilon} \right)$ and $\log (eN) = \Theta \left(\log \left(1 + \frac{LD}{\epsilon}\right) \right)$ in the regime $\frac{LD}{\epsilon} \gg 1$, this proves~\eqref{eq:streaming-complexity}.

\clearpage
\section{Experimental details}
\label{app:experiments}

We further specify the tuning and budget normalization of \ref{eq:dual} used in Section~\ref{sec:experiments}.

\paragraph{Experiment 1.}
For scheduled \ref{eq:dual}, we search the same stepsize grid as its constant-batch counterpart, $\eta \in \{0.005, \allowbreak 0.01, \allowbreak 0.025, \allowbreak 0.05, \allowbreak 0.1, \allowbreak 0.25, \allowbreak 0.5, \allowbreak 0.75, \allowbreak 1\}$, and $N - 1 \in \{32, \allowbreak 47, \allowbreak 69, \allowbreak 101, \allowbreak 148, \allowbreak 217, \allowbreak 317, \allowbreak 465, \allowbreak 682, \allowbreak 1000\}$.
For each candidate $N$, we use
\[
	B_k = \left\lceil\frac{C}{\sqrt{(N - k) (N - k - 1)}}\right\rceil, \quad k = 0, \ldots, N - 2, \qquad \sum_{k = 0}^{N - 2}B_k \le Q = 20000,
\]
and choose the largest feasible $C > 0$.
This preserves the allocation in~\eqref{eq:optimal-cocoercive} while adhering the oracle budget constraint.
Every candidate pair is evaluated on all three tuning seeds, and the selected values are $\eta = 0.75$ and $N = 1001$.
The constant-batch counterpart selects $\eta = 0.025$ and $B = 2$, giving $N = 10001$ and the same total oracle cost.

\paragraph{Experiment 2.}
For the practical variant of \ref{eq:dual} (scheduled), we set
\[
	B_k = \max\! \left\{1, \left\lceil C\frac{\sqrt{\log (eN / (N - k))}}{N - k} \right\rceil \right\}, \quad k = 0, \ldots, N - 1, \quad (m + 1) \sum_{k = 0}^{N - 1}B_k \le Q = 2 \times 10^5.
\]
For each candidate $(m, N)$, we choose the largest feasible $C$.
The constant-batch counterpart uses $B = \left\lfloor \frac{Q}{(m + 1) N} \right\rfloor$.
For both versions, we test $m \in \{1, 2, 4, 7\}$ and $N \in \{1200, \allowbreak 2400, \allowbreak 4800, \allowbreak 9600\}$ on a tuning seed.
We then fix $m = 1$ and refine $N \in \{4800, \allowbreak 9600, \allowbreak 16000, \allowbreak 24000, \allowbreak 40000\}$ for the scheduled method and $N \in \{1600, \allowbreak 2400, \allowbreak 2500, \allowbreak 3200, \allowbreak 3333, \allowbreak 4800, \allowbreak 5000\}$ for the constant-batch method, using all three tuning seeds.
The selected scheduled run has $N = 24000$, and the selected constant run has $N = 3200$ and $B = 31$.

For the theory-prescribed version of scheduled \ref{eq:dual}, we set $\rho_k = \frac{1}{2N^3}$ and use the schedule $B_k$ as above, and choose largest $C$ satisfying $\sum_{k = 0}^{N - 1}\sum_{j = 0}^{B_k - 1} (m_{k, j} + 1) \le Q$, with $m_{k, j}$ given by Algorithm~\ref{alg:resolvent}.
The final search uses $N \in \{1400, 1600, \ldots, 3200\}$ on all tuning seeds and selects $N = 2400$.

\end{document}